\documentclass[11pt]{article}
\usepackage[utf8]{inputenc}

\usepackage[margin=1in]{geometry}

\usepackage{microtype}
\usepackage{graphicx}
\usepackage{subfigure}
\usepackage{booktabs} %

\usepackage{amsmath}

\usepackage[colorlinks=true,allcolors=blue,hyperfootnotes=false]{hyperref}
\usepackage[capitalise]{cleveref}

\newenvironment{acks}{\subsection*{Acknowledgements}}

\title{Efficient Online Linear Control with \\ Stochastic Convex Costs and Unknown Dynamics}

\author{%
Asaf Cassel%
\thanks{School of Computer Science, Tel Aviv University; \texttt{acassel@mail.tau.ac.il}.}
\and
Alon Cohen%
\thanks{School of Electrical Engineering, Tel Aviv University, and Google Research, Tel Aviv; \texttt{alonco@tauex.tau.ac.il}.}
\and
Tomer Koren%
\thanks{School of Computer Science, Tel Aviv University, and Google Research, Tel Aviv; \texttt{tkoren@tauex.tau.ac.il}.}
}

\usepackage[normalem]{ulem}

\RequirePackage{mathtools}
\RequirePackage{dsfont}
\RequirePackage{delimset}

\newcommand{\floor}[2][*]{\delim\lfloor\rfloor#1{#2}}
\newcommand{\indEvent}[2][*]{\mathds{1}_{\brk[c]#1{#2}}}

\newcommand{\tr}[2][*]{\mathrm{Tr}\brk*{#2}}

\newcommand{\tran}{^{\mkern-1.5mu\mathsf{T}}}

\newcommand{\EE}[1][]{\mathbb{E}_{#1}}

\newcommand{\RR}[1][]{\mathbb{R}^{#1}}

\DeclareMathOperator*{\argmin}{arg\,min}

\DeclarePairedDelimiterX\setDef[1]\lbrace\rbrace{#1}
\newcommand{\seqDef}[3]{\brk[c]*{#1}_{#2}^{#3}}

\usepackage{amsthm}
\usepackage{amssymb}
\usepackage{thmtools}

\declaretheoremstyle[
	    spaceabove=\topsep, 
	    spacebelow=\topsep, 
	    headfont=\normalfont\bfseries,
	    bodyfont=\normalfont\itshape,
	    notefont=\normalfont\bfseries,
	    notebraces={(}{)},
	    postheadspace=0.33em, 
	    headpunct={.},
    ]{theorem}
\declaretheorem[style=theorem]{theorem}

\declaretheoremstyle[
	    spaceabove=\topsep, 
	    spacebelow=\topsep, 
	    headfont=\normalfont\bfseries,
	    bodyfont=\normalfont,
	    notefont=\normalfont\bfseries,
	    notebraces={(}{)},
	    postheadspace=0.33em, 
	    headpunct={.},
    ]{definition}

\declaretheoremstyle[
        spaceabove=\topsep, 
        spacebelow=\topsep, 
        headfont=\normalfont\bfseries,
        bodyfont=\normalfont,
        notefont=\normalfont\bfseries,
        notebraces={}{},
        postheadspace=0.33em, 
        qed=$\blacksquare$, 
        headpunct={.},
    ]{proofstyle}
\declaretheorem[style=proofstyle,numbered=no,name=Proof]{proof}

\declaretheorem[style=theorem,sibling=theorem,name=Lemma]{lemma}

\declaretheorem[style=theorem,numbered=no,name=Theorem]{theorem*}
\declaretheorem[style=theorem,numbered=no,name=Lemma]{lemma*}
\declaretheorem[style=theorem,numbered=no,name=Corollary]{corollary*}
\declaretheorem[style=theorem,numbered=no,name=Proposition]{proposition*}
\declaretheorem[style=theorem,numbered=no,name=Claim]{claim*}
\declaretheorem[style=theorem,numbered=no,name=Fact]{fact*}
\declaretheorem[style=theorem,numbered=no,name=Observation]{observation*}
\declaretheorem[style=theorem,numbered=no,name=Conjecture]{conjecture*}

\declaretheorem[style=definition,sibling=theorem,name=Definition]{definition}

\declaretheorem[style=definition,numbered=no,name=Definition]{definition*}
\declaretheorem[style=definition,numbered=no,name=Remark]{remark*}
\declaretheorem[style=definition,numbered=no,name=Example]{example*}
\declaretheorem[style=definition,numbered=no,name=Question]{question*}

\newcommand{\pOCOoptimism}{\alpha}
\newcommand{\at}[1][t]{a_{#1}}
\newcommand{\yt}[1][t]{y_{#1}}
\newcommand{\lt}[1][t]{\ell_{#1}}
\newcommand{\ocoModel}{Q}
\newcommand{\ocoEstModel}{\smash{\widehat{Q}}}
\newcommand{\ocoTrueModel}{Q_{\star}}

\newcommand{\ocoCostVar}{\sigma_{\lt[]}}
\newcommand{\ocozz}[1][]{\zeta_{#1}}
\newcommand{\ocowt}[1][t]{w_{#1}}
\newcommand{\ocoSet}{\mathcal{S}}
\newcommand{\ocoDiam}{R_a}
\newcommand{\din}{d_a}
\newcommand{\dout}{d_y}
\newcommand{\ocoMaxNoise}{W}
\newcommand{\ocoModelDiam}{R_Q}

\newcommand{\ocoR}[1]{R_{#1}}
\newcommand{\LL}{\mu}
\newcommand{\qq}{q}
\newcommand{\LLofu}[1][t]{\bar{\mu}_{#1}}
\newcommand{\ocoTotDiam}{D_q}

\newcommand{\ut}[1][t]{u_{#1}}
\newcommand{\xt}{x_t}
\newcommand{\xtpi}{x_t^{\pi}}
\newcommand{\utpi}{u_t^{\pi}}

\newcommand{\maxF}{C_{f}}
\newcommand{\pLipW}{G_w}

\newcommand{\tauI}[1][i]{\tau_{#1,1}}
\newcommand{\tauIJ}[1][j]{\tau_{i,#1}}
\newcommand{\pRegTheta}{\lambda_\model}
\newcommand{\pRegW}{\lambda_w}
\newcommand{\pOptimism}{\alpha}

\newcommand{\dx}{d_x}
\newcommand{\du}{d_u}

\newcommand{\Astar}{A_{\star}}
\newcommand{\Bstar}{B_{\star}}

\newcommand{\RxuMax}{R_{\max}}
\newcommand{\nEpochs}{N}
\newcommand{\nSubEpochs}[1][i]{N_{#1}}
\newcommand{\wErr}{C_w}

\newcommand{\zt}[1][t]{\zeta_{#1}}
\newcommand{\costVar}{\sigma_c}
\newcommand{\maxNoise}{W}

\newcommand{\RM}{R_{\mathcal{M}}}
\newcommand{\Bbound}{R_B}

\newcommand{\ww}{{w}}
\newcommand{\wwhat}{\hat{w}}

\newcommand{\obsOp}{P}
\newcommand{\obs}{\rho}

\newcommand{\tauIt}{\tau_{i(t),1}}

\newcommand{\wMin}{\underline{\sigma}}

\newcommand{\model}{\Psi}

\usepackage{enumitem}
\usepackage[numbers]{natbib}
\usepackage{crossreftools}
\usepackage{algorithm}
\usepackage[noend]{algpseudocode}

\algnewcommand{\IfThenElse}[3]{%
  \State \algorithmicif\ #1\ \algorithmicthen\ #2\ \algorithmicelse\ #3}

\begin{document}

\maketitle

\begin{abstract}
    We consider the problem of controlling an unknown linear dynamical system under a stochastic convex cost and full feedback of both the state and cost function.
    We present a computationally efficient algorithm that attains an optimal $\sqrt{T}$ regret-rate compared to the best stabilizing linear controller in hindsight.
    In contrast to previous work, our algorithm is based on the Optimism in the Face of Uncertainty paradigm.
    This results in a substantially improved computational complexity and a simpler analysis.
\end{abstract}

\section{Introduction}

Adaptive control, the task of regulating an unknown linear dynamical system, is a classic control-theoretic problem that has been studied extensively since the 1950s~ \citep[e.g.,][]{bertsekas1995dynamic}.
Classic results on adaptive control typically pertain to the asymptotic stability and convergence to the optimal controller while contemporary research focuses on regret minimization and finite-time guarantees.

In linear control, both the state and action are vectors in Euclidean spaces. At each time step, the controller views the current state of the system, chooses an action, and the system transitions to the next state. The latter is chosen via a linear mapping from the current state and action and is perturbed by zero-mean i.i.d.\ noise. The controller also incurs a cost as a function of the instantaneous state and action. 
In classic models, such as the Linear-Quadratic Regulator (LQR), the cost function is quadratic. A fundamental result on LQR states that, when the model parameters are known, the policy that minimizes the steady-state cost takes a simple linear form; namely, that of a fixed linear transformation of the current state~\citep[see][]{bertsekas1995dynamic}.
In more modern formulations, the cost can be any convex Lipschitz function of the state-action pair, and the controller has a no-regret guarantee against the best fixed linear policy~\citep[e.g.,][]{agarwal2019online,agarwal2019logarithmic,simchowitz2020improper,cassel2020bandit}.

In this paper we study linear control in a challenging setting of unknown dynamics and unknown stochastic (i.i.d.) convex costs.
For the analogous scenario in tabular reinforcement learning, efficient and rate-optimal regret minimization algorithms are well-known~\citep[e.g.,][]{auer2008near}. However, similar results for adaptive linear control seem significantly more difficult to obtain.
Prior work in this context has established efficient $\sqrt{T}$-regret algorithms that are able to adapt to adversarially varying convex costs~\citep{agarwal2019online}, but assumed \emph{known} dynamics. 
\cite{simchowitz2020improper} extended this to achieve a $T^{2/3}$-regret for unknown dynamics by using a simple explore-then-exploit strategy:
in the exploration phase, the controller learns the transitions by injecting the system with random actions; in the exploitation phase, the controller runs the original algorithm using the estimated transitions.
We remark that \cite{simchowitz2020improper} also showed that their explore-then-exploit strategy achieves a $\sqrt{T}$ regret bound in this setting for \emph{strongly convex} (adversarial) costs; thus demonstrating that the stringent strong convexity assumption is crucial in allowing one to circumvent the challenge of balancing exploration and exploitation.

Recently, \citet{NEURIPS2020_565e8a41} made progress in this direction. They observed that the problem of learning both stochastic transitions and stochastic convex costs under bandit feedback is reducible
to an instance of stochastic bandit convex optimization for which complex, yet generic polynomial-time algorithms exist~\citep{agarwal2011stochastic}. 
In their case, the bandit feedback assumption requires a brute-force reduction that loses much of the structure of the problem (that would have been preserved under full-feedback access to the costs). This consequently results in a highly complicated algorithm whose running time is a high-degree polynomial in the dimension of the problem (specifically,~$n^{16.5}$).
\citet{NEURIPS2020_565e8a41} also give a more efficient algorithm that avoids a reduction to bandit optimization, but on the other hand assumes the cost function is known and fixed, and that the disturbances in the dynamics come from an isotropic Gaussian distribution.\footnote{\citet{NEURIPS2020_565e8a41} describe how to extend their results to more general noise distributions; however, these distributions would still need to be near-spherical since the algorithm needs to be initialized using a ``warmup'' period in which the dynamics are estimated uniformly.}
Moreover, this algorithm still relies on computationally intensive procedures (for computing barycentric spanners) that involve running the ellipsoid method.

In this work we present a new computationally-efficient algorithm with a $\sqrt{T}$ regret guarantee for linear control with unknown dynamics and unknown stochastic convex costs under full-information feedback.
Our algorithm is simple and intuitive, easily implementable, and works with any sub-Gaussian noise distribution.
It is based on the ``optimism in the face of uncertainty'' (OFU) principle, thought previously to be computationally-infeasible to implement for general convex cost functions~\citep[see][]{NEURIPS2020_565e8a41}.
The OFU approach enables seamless integration between exploration and exploitation, simplifies both algorithm and analysis significantly, and allows for a faster running time by avoiding explicit exploration (e.g, using spanners) in high-dimensional space.

Our OFU implementation is inspired by the well-known UCB algorithm for multi-armed bandits \citep{auer2002finite}. 
That is, we minimize a lower confidence bound that is constructed as the difference between the (convex) empirical loss and an exploration bonus term whose purpose is to draw the policy towards underexplored state-action pairs.
However, since the exploration term is also convex, minimizing the lower confidence bound unfortunately results in a nonconvex optimization problem which, at first glance, can be seen as computationally-hard to solve.
Using a trick borrowed from stochastic linear bandits \citep{dani2008stochastic}, we nevertheless are able to relax the objective in such a way that allows for a polynomial-time solution, rendering our algorithm computationally-efficient~overall.

\paragraph{Related work.}

The problem of adaptive LQR control with known fixed costs and unknown dynamics has had a long history. 
\cite{abbasi2011regret} were the first to study this problem in a regret minimization framework. 
Their algorithm is also based on OFU, and while inefficient, guarantees rate-optimal $\sqrt{T}$ regret albeit with exponential dependencies on the dimensionality of the system.
Since then, many works have tried improving the regret guarantee, \cite{ibrahimi2012efficient,faradonbeh2017finite,dean2018regret} to name a few.
The latter work also presented a poly-time algorithm at a price of a $T^{2/3}$-type regret bound.  \cite{cohen2019learning,mania2019certainty} improve on this by showing how to preserve the $\sqrt{T}$ regret rate with computational efficiency.
The optimality of the $\sqrt{T}$ rate was proved concurrently by \cite{cassel2020logarithmic,simchowitz2020naive}.
\cite{dean2018regret} were the first to assume access to a stabilizing controller in order to obtain regret that is polynomial in the problem dimensions. This was later shown to be necessary by \cite{chen2021black}.

Past work has also considered adaptive LQG control, namely LQR under partial observability of the state \citep[for example,][]{simchowitz2020improper}. 
However, it turned out that (in the stochastic setting) learning the optimal partial-observation linear controller is easier than learning the full-observation controller, and, in fact, it is possible to obtain $\text{poly} \log T$ regret for adaptive LQG \citep{lale2020logarithmic}.

Another line of work, initiated by \cite{cohen2018online}, deals with adversarial LQR in which the transitions are fixed and known, but the cost function changes adversarially. 
\cite{agarwal2019online} extended this setting to adversarial noise as well as arbitrary convex Lipschitz costs. 
Subsequently \cite{agarwal2019logarithmic,foster2020logarithmic,simchowitz2020making} provided a $\text{poly} \log T$ regret guarantee for strongly-convex costs with the latter also handling fully adversarial disturbances.
\cite{cassel2020bandit,gradu2020non} show a $\sqrt{T}$ regret bound for bandit feedback over the cost function.
Lastly, works such as \cite{goel2019online} bound the competitive ratio of the learning algorithm rather than its regret.

In a recent follow-up work \citep{cassel2022rate}, we provide an analogous $\smash{\sqrt{T}}$ regret algorithm for the more challenging case of adversarial cost functions (and unknown dynamics). The result builds on the OFU approach introduced here and combines it with a novel and efficient online algorithm, which minimizes regret with respect to the non-convex optimistic loss functions. In both the stochastic and adversarial cases, the results strongly depend on the stochastic nature of the disturbances; this is in contrast with \cite{simchowitz2020improper,simchowitz2020making}, which consider adversarial costs and disturbances. The first shows a $\smash{T^{2/3}}$ regret algorithm for general convex costs, and the second gives a $\smash{\sqrt{T}}$ regret algorithm for strongly-convex costs. It thus remains open whether $\smash{\sqrt{T}}$ regret can be achieved for adversarial disturbances and general convex costs.

\section{Problem Setup}

We consider controlling an unknown linear dynamical system under stochastic convex costs and full state and cost observation.
Our goal is to minimize the total control cost in the following online setting where at round $t$:
\begin{enumerate}[nosep,label=(\arabic*)]
    \item The player observes state $x_t$;
    \item The player chooses control $u_t$;
    \item The player observes the cost function $c_t : \RR[\dx] \times \RR[\du] \to \RR$, and incurs cost $c_t(x_t, u_t)$;
    \item The system transitions to $x_{t+1} = \Astar x_t + \Bstar u_t + w_t$,
    where $\Astar \in \mathbb{R}^{\dx \times \dx}$, $\Bstar \in \mathbb{R}^{\dx \times \du}$, and $w_t \in \mathbb{R}^{\dx}$.
\end{enumerate}
Our goal is to minimize regret with respect to any policy $\pi$ in a benchmark policy class $\Pi$. %
To that end, denote by $\xtpi, \utpi$ the state and action sequence resulting when following a policy $\pi \in \Pi$. Then the regret is defined as
\begin{align*}
    \mathrm{regret_T(\pi)}
    =
    \sum_{t=1}^{T} c_t(\xt,\ut)
    - c_t(\xtpi,\utpi)
    ,
\end{align*}
and we seek to bound this quantity with high probability for all $\pi \in \Pi$.

To define the policy class $\Pi$, we use the following notion of stability due to \cite{cohen2018online}, which is essentially a quantitative version of classic stability notions in linear control.
\begin{definition}[Strong stability]
    A controller $K$ for the system $(\Astar, \Bstar)$ is $(\kappa, \gamma)-$strongly stable ($\kappa \ge 1$, $0 < \gamma \le 1$) if there exist matrices $Q,L$ such that
    $\Astar +\Bstar K = Q L Q^{-1}$,
    $\norm{L} \le 1 - \gamma$,
    and
    $\norm{K}, \norm{Q}\norm{Q^{-1}} \le \kappa$.
\end{definition}
We consider the benchmark policy class of linear policies that choose $u_t = K x_t$. i.e.,
\begin{align*}
    \Pi_{\mathrm{lin}}
    =
    \brk[c]{
    K \in \RR[\du \times \dx] 
    \; : \;
    \text{$K$ is $(\kappa,\gamma)-$ strongly stable}
    }.
\end{align*}

We make the following assumptions on our learning problem:
\begin{itemize}[leftmargin=*]
    \item \textbf{Bounded stochastic costs:} The cost functions are such that $c_t(x,u) := c(x, u; \zt)$ where $(\zt)_{t=1}^{T}$ is a sequence of i.i.d.\ random variables. Moreover, for all $x,u$,
    $
    \abs{c(x,u; \zt[]) - \EE[{\zt[]'}]c(x,u; \zt[]')}
    \le
    \costVar
    $;
    \item \textbf{Lipschitz costs:} For any $(x, u), (x', u')$ we have
    \begin{align*}
        \abs{c_t(x, u) - c_t(x', u')}
        \le
        \norm{(x - x', u - u')}
        ,
    \end{align*}
    \item \textbf{Bounded i.i.d.\ noise:} $(w_t)_{t=1}^{T}$ is a sequence of i.i.d.\ random variables such that $\norm{w_t} \le \maxNoise$;
    \item \textbf{Lower-bounded covariance:} There exists (an unknown) $\wMin > 0$ such that $\EE w_t w_t\tran \succeq \wMin^2 I$;
    \item \textbf{Stabilizing controller:} $\Astar$ is $(\kappa,\gamma)-$strongly stable, and $\norm{\Bstar} \le \Bbound$.
\end{itemize}

Note the assumption that $A_\star$ is strongly stable is without loss of generality.
Otherwise, given access to a stabilizing controller $K$, \cite{cassel2022rate} show a general reduction, which essentially adds $K x_t$ to our actions. This will replace $\Astar$ in the analysis with $\Astar + \Bstar K$, which is $(\kappa,\gamma)-$strongly stable, as desired. However, this will also add the burden of adding $K x_t$ to our actions throughout the paper, only making for a more taxing and tiresome reading.

We also remark that the bounded noise assumption can be alleviated to sub-Gaussian noise instead, and that (sub-)Quadratic costs can also be accommodated by appropriately rescaling them. This is essentially since both sub-Gaussian noise and the state and action sequences are bounded with high probability (see \citep{cassel2020bandit,cassel2022rate} for more details on these techniques).

\section{Algorithm and main result}

We now present our result for the general linear control problem (i.e., $\Astar \neq 0$).
We begin by giving necessary preliminaries on Disturbance Action Policies, then we provide our algorithm and give a brief sketch of its regret analysis. 
The full details of the analysis are deferred to \cref{sec:proofOfLqrRegret}.

\subsection{Preliminaries: Disturbance Action Policies (DAP)}
Following recent literature, we use the class of Disturbance Action Policies first proposed by \cite{agarwal2019online}. This class is parameterized by a sequence of matrices $\brk[c]{M^{[h]} \in \RR[\du \times \dx]}_{h=1}^{H}$. For brevity of notation, these are concatenated into a single matrix $M \in \RR[\du \times H \dx]$ defined as
\begin{align*}
    M
    =
    \brk1{
    M^{[1]}
    \cdots
    M^{[H]}
    }
    .
\end{align*}
A Disturbance Action Policy $\pi_M$ chooses actions
\begin{align*}
    u_t = \sum_{h=1}^{H} M^{[h]} w_{t-h}
    ,
\end{align*}
where recall that the $w_t$ are system disturbances.
Consider the benchmark policy class
\begin{align*}
    \Pi_{\mathrm{DAP}}
    =
    \brk[c]*{
    \pi_M
    \;:\;
    \norm{M}_F \le \RM
    }
    .
\end{align*}
We note that there are several ways to define this class with the most common considering $\sum_{h=1}^{H} \norm{M^{[h]}}$ instead of $\norm{M}_F$. We chose the Frobenius norm for simplicity of the analysis and implementation, but replacing it would not change the analysis significantly.

The importance of this policy class is two-fold. First, as shown in Lemma 5.2 of \cite{agarwal2019online}, if 
$H \in \Omega(\gamma^{-1} \log T)$ 
and
$\RM \in \Omega(\kappa^2 \sqrt{\du / \gamma})$ 
then $\Pi_{\mathrm{DAP}}$ is a good approximation for $\Pi_{\mathrm{lin}}$ in the sense that a regret guarantee with respect to $\Pi_{\mathrm{DAP}}$ gives the same guarantee with respect to $\Pi_{\mathrm{lin}}$ up to a constant additive factor. Second, its parameterization preserves the convex structure of the problem, making it amenable to various online convex optimization methods.
In light of the above, our regret guarantee will be given with respect to $\Pi_{\mathrm{DAP}}$.

While the benefits of $\Pi_{\mathrm{DAP}}$ are clear, notice that it cannot be implemented under our assumptions. This is since we do not have access to the system disturbances $w_t$ nor can we accurately recover them due to the uncertainty in the transition model.
Similarly to previous works, our algorithm thus uses estimated disturbances $\hat{w}_t$ to compute its actions. 

\paragraph{Finite memory representation.}

As is common in recent literature, we will approximate the various problem parameters with bounded memory representations. To see this, recurse over the transition model to get that
\begin{align}
\label{eq:lqrUnrolling}
    x_t
    =
    \Astar^{H} x_{t-H}
    +
    \sum_{i=1}^{H} \brk*{
    \Astar^{i-1}\Bstar \ut[t-i]
    +
    \Astar^{i-1} {w}_{t-i}}
    =
    \Astar^{H} x_{t-H}
    +
    \model_\star \tilde{\obs}_{t-1} + {w}_{t-1}
    ,
\end{align}
where
$
\model_\star
=
\brk[s]{\Astar^{H-1}\Bstar, \ldots, \Astar\Bstar, \Bstar, \Astar^{H-1}, \ldots, \Astar}
,
$
and
$
\tilde{\obs}_t
=
\brk[s]{
u_{t-H}\tran, \ldots, u_t\tran,
w_{t-H}\tran, \ldots, w_{t-1}\tran
}\tran
.
$
Now, since $\Astar$ is strongly stable, the term $\Astar^{H} x_{t-H}$ quickly becomes negligible. Combining this with the DAP policy parameterization, we define the following bounded memory representations. For an arbitrary sequence of disturbances $w = \brk[c]{w_t}_{t \ge 1}$ define
\begin{alignat}{2}
    \nonumber
    &u_t(M; \ww)
    &&
    =
    \textstyle\sum_{h=1}^{H} M^{[h]} w_{t-h};
    \\
    \label{eq:bounded-mem-rep}
    &\obsOp(M) 
    &&=
    \begin{pmatrix}
        M^{[H]} & M^{[H-1]} & \cdots & M^{[1]} \\
        & M^{[H]} & M^{[H-1]} & \cdots &  M^{[1]}  \\
         &  & \ddots & \ddots & &  \ddots &  \\
        &  & & M^{[H]} & M^{[H-1]} & \cdots &  M^{[1]} \\
          &  & &  & I &  \\
           &  & &  &  & \ddots \\
          &  & &  & & & I
    \end{pmatrix}
    ;
    \\
    \nonumber
    &\obs_t(M; \ww)
    &&
    =
    \brk{
    u_{t+1-H}(M; \ww)\tran
    ,
    \ldots
    u_{t}(M; \ww)\tran
    ,
    w_{t+1-H}
    ,
    \ldots
    ,
    w_{t-1}
    }\tran
    =
    \obsOp(M) w_{t+1-2H:t-1};
    \\
    \nonumber
    &x_t(M; \model, \ww)
    &&
    =
    \model \obs_{t-1}(M; \ww) + w_{t-1}
    .
\end{alignat}
Notice that $u_t, \obs_t, x_t$ do not depend on the entire sequence $\ww$, but only $w_{t-H:t-1}, w_{t+1-2H:t-1},$ and $w_{t-2H:t-1}$ respectively. Importantly, this means that we can compute these functions with knowledge of only the last (at most) $2H$ disturbances. While our notation does not reveal this fact explicitly, it helps with both brevity and clarity.

\subsection{Algorithm}

\begin{algorithm}[t]
	\caption{Stochastic Linear Control Algorithm} \label{alg:lqr}
	\begin{algorithmic}[1]
		\State \textbf{input}:
		memory length $H$, optimism parameter $\pOptimism$, regularization parameters $\lambda_\model, \lambda_w$.
		
		\State \textbf{set} $i=j=1, \tauI[1] = 1, V_1 = \lambda_\model I, M_1 = 0$ and $\hat{w}_t = 0, u_t = 0$ for all $t < 1$.

	    \For{$t = 1,2,\ldots, T$}
	        
	        \State \textbf{play} 
	        $
	            \ut
	            =
	            \sum_{h=1}^{H} M_{t}^{[h]} \hat{w}_{t-h}
	        $
	        where
	        $M_t = M_{\tauIJ}$
	        \label{ln:control-choice}
	        
	        \State \textbf{observe} $x_{t+1}$ and cost function $c_t$.
	        \State \textbf{calculate}
	        \begin{align*} 
	            (A_t \; B_t)
	            =
	            \argmin_{(A \; B) \in \mathbb{R}^{\dx \times (\dx + \du)}} 
	            \sum_{s=1}^{t} \norm{(A \; B) z_s - x_{s+1}}^2
	            +
	            \lambda_w \norm{(A \; B)}_F^2
	            ,
	            \quad 
	            \text{where}
	            \;
	            z_s = \begin{pmatrix}x_s \\ u_s \end{pmatrix}.
	        \end{align*}
	        \label{ln:est-ab}
	        \State \textbf{set}
	        $V_{t+1} = V_t + \obs_t \obs_t\tran$
	        for 
	        $\obs_t = \brk{u_{t+1-H}\tran, \ldots, u_{t}\tran, \hat{w}_{t+1-H}\tran, \ldots, \hat{w}_{t-1}\tran}\tran$
	        .
	        \State {\bf estimate noise}
	        $
	        \hat{w}_t 
	        =
	        \Pi_{B_2(\maxNoise)}\brk[s]{x_{t+1} - A_t x_t - B_t u_t}
	        $.
	        \label{ln:est-noise}
	        \If{$\det(V_{t+1}) > 2 \det(V_{\tauI})$} \label{ln:det-double}
	            \State \textbf{start new epoch}: $i = i+1, j = 2, \tauI = t+1, \tauIJ[2] = \tauI + 2H, M_{\tauIJ[1]} = M_{\tauIJ[2]}=0$.
	            \State \textbf{estimate system parameters}
	            \begin{align*} 
	                \model_{\tauI}
	                =
	                \argmin_{\model \in \mathbb{R}^{\dx \times (H \du + (H-1) \dx)}}
	                \brk[c]*{
	                    \sum_{s = 1}^{t} \norm{\model \obs_{s} - x_{s+1}}^2 + \lambda_{\model}\norm{\model}^2
	                }
	                .
                \end{align*}
                \label{ln:est-params}
            \EndIf
            \If{$t+1-\tauI > 2(\tauIJ - \tauI)$} \label{ln:subepochs}
                \State \textbf{start new sub-epoch}: $j = j + 1, \tauIJ = t+1$.
                \State \textbf{solve optimistic cost minimization}
                ($
                \wwhat
                =
                \brk[c]{\hat{w}_t}_{t \ge 1}
                $)
                \label{ln:opt-cost-min}
	            \begin{align*}
	                M_{\tauIJ}
	                =
	                \argmin_{M \in \mathcal{M}}
	                \sum_{s = \tauIJ[j-1]}^{\tauIJ-1} \brk[s]*{
	                c_s({x}_s(M; \model_{\tauI}, \wwhat), u_s(M; \wwhat)))
	                -
	                \pOptimism \maxNoise \norm{V_{\tauI}^{-1/2} \obsOp(M)}_\infty
	                }
	                .
	            \end{align*}
	        \EndIf
	    \EndFor
	\end{algorithmic}
\end{algorithm}

Here we present \cref{alg:lqr} for general linear systems ($\Astar \neq 0$). 
Notice that the system's memory as well as the use of DAP policies with the estimated noise terms $(\hat w_t)$ can cause for cyclical probabilistic dependencies between the estimate of the model transitions, the estimate of the loss, and the estimated noise terms. 
To alleviate these dependencies our algorithm seldom changes its chosen policy ($\approx \log^2 T$ many times), and constructs its estimates using only observations from previous non-overlapping time intervals.

The algorithm proceeds in epochs, each starting with a least squares estimation of the unrolled model using all past observations (\cref{ln:est-params}), and the estimate is then kept fixed throughout the epoch. The epoch ends when the determinant of $V_t$ is doubled (\cref{ln:det-double}); intuitively, when the confidence of the unrolled model increases substantially.\footnote{More concretely, the volume of the confidence ellipsoid around the unrolled model decreases by a constant factor.} 
An epoch is divided into subepochs of exponentially growing lengths in which the policy is kept fixed (\cref{ln:subepochs}).
Each subepoch starts by minimizing an optimistic estimate of the loss (\cref{ln:opt-cost-min}) that balances between exploration.
The algorithm plays the resulting optimistic policy throughout the subepoch (\cref{ln:control-choice}). To that end we follow the technique presented in \cite{NEURIPS2020_565e8a41} to estimate the noise terms $\brk[c]0{\ocowt}_{t \ge 1}$ on-the-fly (\cref{ln:est-noise}). Note that, for this purpose, the algorithm estimates the matrix $(\Astar \; \Bstar)$ in each time step (\cref{ln:est-ab}) even though it can be derived from the estimated unrolled model. This is done to simplify our analysis, and only incurs a small price on the runtime of the algorithm.

We have the following guarantee for our algorithm:

\begin{theorem}
\label{thm:lqrRegret}
    Let $\delta \in (0,1)$ and suppose that we run \cref{alg:lqr} with parameters $\RM, \Bbound \ge 1$ and for proper choices of $H, \pRegW, \pRegTheta, \pOptimism$.
    If 
    $
    T
    \ge
    64 \RM^2
    $
    then with probability at least $1-\delta$, simultaneously for all $\pi \in \Pi_{\mathrm{DAP}}$,
    \begin{align*}
        \mathrm{regret}_T(\pi)
        \le
        \mathrm{poly}(\kappa, \gamma^{-1}, \wMin^{-1}, \costVar, \Bbound, \RM, \dx, \du, \log(T/\delta))
        \sqrt{ T}
        .
    \end{align*}
\end{theorem}

\paragraph{Efficient computation.} 

The main hurdle towards computational efficiency is the calculation of the optimistic cost minimization step (\cref{ln:opt-cost-min}).
We compute this in polynomial-time by borrowing a trick from \cite{dani2008stochastic}: the algorithm solves $2m$ convex optimization problems with $m=\dx (2H-1) (\dx (H-1) + \du H)$, and takes the minimum between them.
To see why this is valid, observe that $\norm{x}_\infty = \max_{\chi \in \{-1,1\}} \max_{k \in [m]} \chi \cdot x_k$. 
We can therefore write the optimistic cost minimization as 
\[
    \min_{\chi \in \{-1,1\}
    ,
    k \in [m]}
    \min_{M \in \mathcal{M}}
    \sum_{s = \tauIJ[j-1]}^{\tauIJ-1} \brk[s]*{
	                c_s({x}_s(M; \model_{\tauI}, \wwhat), u_s(M; \wwhat)))
	                -
	                \pOptimism \maxNoise \chi \cdot \brk*{V_{\tauI}^{-1/2} \obsOp(M)}_k
	                }
	                ,
\]
where $k$ is a linear index.
This indeed suggests to solve for $M \in \mathcal{M}$ for each value of $k$ and $\chi$, then take the minimum between them.
Moreover, when $k$ and $\chi$ are fixed, the objective becomes convex in $M$.
Consequently, 
As there are $2 m$ such values of $k$ and $\chi$, this amounts to solving $2 m$ convex optimization problems. We note that it suffices to solve each convex optimization problem up to an accuracy of $\approx T^{-1/2}$, which can be done using $O(T)$ gradient oracle calls.

\paragraph{Comparison with \cite{NEURIPS2020_565e8a41}.}
The following compares the computational complexity of \cref{alg:lqr} with those of \cite{NEURIPS2020_565e8a41} under a first order (value and gradient) oracle assumption on the cost functions $c_t$. 
To simplify the discussion, we denote both state and action dimensions as $d = \max\brk[c]{\dx, \du}$, and omit logarithmic terms and $O(\cdot)$ notations from all bounds.

We show that the overall computational complexity of our algorithm is $d^4 T$.
By updating the least squares procedure in \cref{ln:est-ab,ln:est-params} recursively at each time step their overall complexity is $d^3 T$.
As previously explained, in \cref{ln:optimisticMin}, we solve $d^2$ convex optimization problems, each to an accuracy of $\approx T^{-1/2}$. Since the objective is a sum of convex functions, we can do this using Stochastic Gradient Descent (SGD) with $T$ oracle calls (in expectation). Overall, we make $d^2 T$ gradient oracle calls. For each oracle call, we further use matrix addition, and matrix vector multiplications on $M$, which take an additional $d^2$ computations. The remaining computations of the algorithm are negligible.

Now, we show that the equivalent complexity for Algorithm 2 of \cite{NEURIPS2020_565e8a41} is $d^{12} T$. The crux of their computation is finding a $2-$Barycentric spanner for their confidence set. The inherent dimension there is that of $M$, which is $d^2$. To compute the barycentric spanner, the authors explain that $d^4$ calls to a linear optimization oracle are required. These can in turn be implemented using the elipsoid method, whose overall computation is $d^8 T$.

\cite{NEURIPS2020_565e8a41} also give Algorithm 5, which works with bandit feedback and uses SBCO as a black-box. Compared with their Algorithm 2 the computational complexity is higher, and the regret guarantee depends on $d^{36}$ instead of $d^3$.

\section{Analysis}

In this section we give a (nearly) complete proof of \cref{thm:lqrRegret} in a simplified setup, inspired by \cite{NEURIPS2020_565e8a41}, where $\Astar = 0$. At the end of the section, we give an overview of the analysis for the general control setting (with $\Astar \neq 0$). The complete details in the general case are significantly more technical and thus deferred from this extended abstract (see \cref{sec:proofOfLqrRegret} for full details).

Concretely, following \cite{NEURIPS2020_565e8a41}, suppose that $\Astar = 0$, and thus $x_{t+1} = \Bstar u_t + w_t$. Next, assume that $c_t(x, u) = c_t(x)$, i.e., the costs do not depend on $u$. Finally, assume that we minimize the pseudo regret, i.e.,
\begin{align*}
    \max_{u : \norm{u} \le R_u}
    \sum_{t=1}^{T} \brk[s]{
    J(\Bstar u_t) - J(\Bstar u)
    }
    ,
\end{align*}
where
$J(\Bstar u) = \EE[{\zt[]}, w]c(\Bstar u + w; \zt[])$.
This setting falls under the umbrella of stochastic bandit convex optimization, making generic algorithms applicable. However, it has additional structure that we leverage to create a much simpler and more efficient algorithm. In what follows, we formally define this setting with clean notation as to avoid confusion with our general setting.

\subsection{The \texorpdfstring{$\boldsymbol{\Astar = 0}$}{A = 0} case: Stochastic Convex Optimization with a Hidden Linear Transform}
Consider the following setting of online convex optimization. Let $\ocoSet \subseteq \RR[\din]$ be a convex decision set. At round $t$ the learner
\begin{enumerate}[label=(\arabic*),nosep]
    \item predicts $\at \in \ocoSet$;
    \item observes cost function $\lt: \RR[\dout] \to \RR$ and state $\yt[t+1] := \ocoTrueModel \at + \ocowt$;
    \item incurs cost $\lt(\ocoTrueModel \at)$.
\end{enumerate}

We have that $\ocowt \in \RR[\dout]$ are i.i.d.\ noise terms, $\ocoTrueModel \in \RR[\dout \times \din]$ is an unknown linear transform, and $\yt \in \RR[\dout]$ are noisy observations. 

The cost functions are stochastic in the following sense. There exists a sequence $\ocozz[1], \ocozz[2], \ldots$ of i.i.d.\ random variables, and a function $\lt[] : \RR[\dout] \times \RR \to \RR$ such that
$
    \lt(\qq) := \lt[](\qq; \ocozz[t])
    .
$
Define the expected cost
$
    \LL(\qq)
    =
    \EE[\ocozz] \lt[](\qq; \ocozz)
    .
$
We consider minimizing the pseudo-regret, defined as
\begin{align*}
    \text{regret}_T
    =
    \max_{\at[] \in \ocoSet}
    \sum_{t=1}^{T} \brk[s]{
    \LL(\ocoTrueModel \at)
    -
    \LL(\ocoTrueModel \at[])
    }
    .
\end{align*}
Minimizing the pseudo-regret instead of the actual regret will maintain the main hardness of the problem, but will better highlight our main contributions.

\paragraph{Assumptions.}
Our assumptions in the simplified case are the following:
\begin{itemize}[nosep,leftmargin=*]
    \item $\lt[](\qq; \ocozz)$ is convex and $1-$Lipschitz in its first parameter;
    \item For all $\ocozz, \ocozz'$, and any $\qq$ we have
    $
    \abs{
    \lt[](\qq; \ocozz)
    -
    \lt[](\qq; \ocozz')
    }
    \le
    \ocoCostVar
    ;
    $
    \item There exists some known $\ocoMaxNoise, \ocoModelDiam \ge 0$ such that $\norm{\ocowt} \le \ocoMaxNoise$, and $\norm{\ocoTrueModel} \le \ocoModelDiam$.
    \item The diameter of $\ocoSet$ is $\ocoDiam = \max_{\at[],\at[]' \in \ocoSet} \norm{\at[] - \at[]'} < \infty$.
\end{itemize}

\subsection{The simplified algorithm}

\begin{algorithm}[ht]
	\caption{SCO with hidden linear transform \label{alg:SCO}}
	\begin{algorithmic}[1]
		\State \textbf{input:}
		optimism parameter $\pOCOoptimism$, regularizer $\lambda$
		
		\State \textbf{set:} $V_1 = \lambda I, \ocoEstModel_1 = 0$.

	    \For{$t = 1,2,\ldots, T$}
	        
	        \State \textbf{play} optimistic cost minimizer:
                $
	                \at
	                \in
	                \argmin_{\at[] \in \mathcal{S}}
	                \sum_{s = 1}^{t-1}
	                \brk[s]{
	                \lt[s](\ocoEstModel_t \at[])
	                -
	                \pOCOoptimism \norm{V_{t}^{-1/2} \at[]}_\infty
	                }.
	            $
	            \label{ln:optimisticMin}
	        \State 
	        \textbf{observe} $\yt[t+1]=\ocoTrueModel a_t + w_t$ and cost function $\lt$, and \textbf{set}
	        $V_{t+1} = V_t + \at \at\tran$.

	            \State 
	            \textbf{estimate} system parameters
	            \begin{align*}
	                \ocoEstModel_{t+1}
	                =
	                \argmin_{\ocoModel \in \RR[\dout \times \din]} \sum_{s = 1}^{t} \norm{\ocoModel \at[s] - \yt[s+1]}^2 + \lambda \norm{\ocoModel}_F^2
	                .
                \end{align*}
	    \EndFor
	\end{algorithmic}
\end{algorithm}

Our algorithm is depicted as \cref{alg:SCO}.
The algorithm maintains an estimate $\ocoEstModel_t$ of $\ocoTrueModel$.
At each time step $t$, the algorithm plays action $\at$, chosen such that it minimizes our optimistic cost function.
That is, $\at$ is a minimizer of a lower bound on the total loss up to time $t$: $\sum_{s=1}^{t-1} \ell_s(\ocoTrueModel a) \approx (t-1) \mu(\ocoTrueModel a)$. 
This fuses together exploration and exploitation by either choosing under-explored actions, or exploiting low expected cost for already sufficiently-explored actions.
We remark that the optimistic cost minimization procedure solves a non-convex optimization problem, but nevertheless show that it can be solved in polynomial-time in the sequel. 
Lastly, our algorithm observes $y_{t+1}$ and uses it improve its estimate of $\ocoTrueModel$ by solving a least-squares problem.

The main hurdle in understanding why the algorithm is computationally-efficient is the calculation of the optimistic cost minimization step. Following the computational method presented in \cref{alg:lqr}, we do this by first solving $2 \din$ convex objectives and then taking their minimizer.

\subsection{Analysis}

We now present the main theorem for this section that bounds the regret of \cref{alg:SCO} with high probability.

\begin{theorem}
\label{thm:scoRegret}
    Let $\delta \in (0,1)$ and suppose that we run \cref{alg:SCO} with parameters
    \begin{align*}
        \lambda = \ocoDiam^2
        ,
        \quad
        \pOCOoptimism
        =
        \sqrt{\din} \brk2{
        \ocoMaxNoise \dout\sqrt{8 \log\tfrac{2T}{\delta}}
        +
        \sqrt{2}\ocoDiam \ocoModelDiam}
        .
    \end{align*}
    If 
    $
    T
    \ge
    \max\brk[c]{\ocoCostVar, 64 \ocoTotDiam^2}
    $
    where
    $
    \ocoTotDiam
    =
    3\ocoModelDiam \ocoDiam
    +
    {\ocoMaxNoise \dout} \sqrt{8 \log{\tfrac{4T}{\delta}}}
    $
    then with probability at least $1-\delta$,
    \begin{align*}
        \text{regret}_T
        \le
        13 \brk[s]*{
        \ocoCostVar\sqrt{\dout \log \frac{3 T}{\ocoCostVar \delta}}
        +
        \din
        \brk{
        \ocoMaxNoise \dout
        +
        \ocoDiam \ocoModelDiam
        }\log\frac{2T}{\delta}
        }\sqrt{T}
        .
    \end{align*}
\end{theorem}

At its core, the proof of \cref{thm:scoRegret} employs 
the Optimism in the Face of Uncertainty (OFU) approach. To that end, define the optimistic cost functions
\begin{align*}
    \LLofu(\at[])
    =
    \LL(\ocoEstModel_t \at[]) - \pOCOoptimism \norm{V_t^{-1/2} \at[]}_\infty
    ,
\end{align*}
where $\ocoEstModel_t, V_t$ are defined as in \cref{alg:SCO}. The following lemma shows that our optimistic loss lower bounds the true loss, and bounds the error between the two.
\begin{lemma}
\label{lemma:ocoOptimismBound}
    Suppose that
    $
    \sqrt{\din} \norm{\ocoEstModel_t - \ocoTrueModel}_{V_t}
    \le
    \pOCOoptimism
    .
    $
    Then we have that
    \begin{align*}
        \LLofu(\at[])
        \le
        \LL(\ocoTrueModel \at[])
        &
        \le
        \LLofu(\at[]) + 2 \pOCOoptimism \sqrt{\at[]\tran V_t^{-1} \at[]}
        .
    \end{align*}
\end{lemma}
\begin{proof}
We first use the Lipschitz assumption to get
\begin{align*}
    \abs{\LL(\ocoTrueModel \at[]) - \LL(\ocoEstModel_t \at[])}
    &
    \le
    \norm{(\ocoTrueModel - \ocoEstModel_t) \at[]}
    \\
    &
    \le
    \norm{\ocoTrueModel - \ocoEstModel_t}_{V_t} \norm{V_t^{-1/2} \at[]}
    \\
    &\le
    \frac{\pOCOoptimism}{\sqrt{\din}} \norm{V_t^{-1/2} \at[]}
    \\
    &\le
    {\pOCOoptimism} \norm{V_t^{-1/2} \at[]}_\infty
    ,
\end{align*}
where the second and third transitions also used the estimation error and that $\norm{a} \le \sqrt{\din} \norm{a}_\infty$.
We thus have on one hand,
\begin{align*}
    \LL(\ocoTrueModel \at[])
    \ge
    \LL(\ocoEstModel_t \at[])
    -
    {\pOCOoptimism} \norm{V_t^{-1/2} \at[]}_\infty
    =
    \LLofu(\at[])
    ,
\end{align*}
and on the other hand we also have
\begin{align*}
    \LL(\ocoTrueModel \at[])
    &
    \le
    \LL(\ocoEstModel_t \at[])
    +
    {\pOCOoptimism} \norm{V_t^{-1/2} \at[]}_\infty
    =
    \LLofu(\at[])
    +
    2{\pOCOoptimism} \norm{V_t^{-1/2} \at[]}_\infty
    \le
    \LLofu(\at[])
    +
    2{\pOCOoptimism} \sqrt{\at[]\tran V_t^{-1} \at[]}
    ,
\end{align*}
where the last step also used $\norm{a}_\infty \le \norm{a}$.
\end{proof}
We are now ready to prove \cref{thm:scoRegret}. The proof focuses on the main ideas, deferring some details to \cref{sec:ocoAdditionalDetails}.
\begin{proof}[of \cref{thm:scoRegret}]
First, notice that
$
\abs{\LL(\ocoTrueModel \at) - \LL(\ocoTrueModel \at[])}
\le
\norm{\ocoTrueModel} \norm{\at - \at[]}
\le
2 \ocoModelDiam \ocoDiam
.
$
Using this bound for $t=1$, we can decompose the regret as
\begin{align*}
    \text{regret}(\at[])
    \le
    2 \ocoModelDiam \ocoDiam
    +
    \underbrace{
    \sum_{t=2}^{T} \LL(\ocoTrueModel \at) - \LLofu(\at)
    }_{\ocoR{1}}
    +
    \underbrace{
    \sum_{t=2}^{T} \LLofu(\at) - \LLofu(\at[])
    }_{\ocoR{2}}
    +
    \underbrace{
    \sum_{t=2}^{T} \LLofu(\at[]) - \LL(\ocoTrueModel \at[])
    }_{\ocoR{3}}
    .
\end{align*}
We begin by bounding $\ocoR{1}$ and $\ocoR{3}$, which relate the true loss to its optimistic variant. To that end, we use a standard least squares estimation bound (\cref{lemma:ocoParameterEst}) to get that \cref{lemma:ocoOptimismBound} holds with probability at least $1 - \delta / 2$.
Conditioned on this event, we immediately get $\ocoR{3} \le 0$. Moreover, we get
\begin{align*}
    \ocoR{1}
    &
    \le
    \sum_{t=2}^{T} 2 \pOCOoptimism \sqrt{\at\tran V_t^{-1} \at}
    \le
    2 \pOCOoptimism \sqrt{T \sum_{t=1}^{T}\at\tran V_t^{-1} \at}
    \le
    2 \pOCOoptimism \sqrt{5 T \din \log T}
    ,
\end{align*}
where the second inequality is due to Jensen's inequality, and the third is a standard algebraic argument (\cref{lemma:harmonicBound}).

Next, we bound $\ocoR{2}$, which is the sum of excess risk of $\at$ with respect to the optimistic cost. 
To that end, we first bound $\norm{\ocoEstModel_t}$, using the least squares error bound (\cref{lemma:ocoParameterEst}) and $V_t \succeq \lambda I = \ocoDiam^2 I$ to get that
\begin{align*}
    \norm{\ocoEstModel_t}
    \le
    \norm{\ocoTrueModel}
    +
    \norm{\ocoEstModel_t - \ocoTrueModel}
    \le
    \ocoModelDiam
    +
    \frac{1}{\ocoDiam} \norm{\ocoEstModel_t - \ocoTrueModel}_{V_t}
    \le
    3\ocoModelDiam
    +
    \frac{\ocoMaxNoise \dout}{\ocoDiam} \sqrt{8 \log \frac{4T}{\delta}}.
\end{align*}
We thus have for all $\at[] \in \ocoSet$,
\begin{align*}
    \norm{\ocoEstModel_t \at[]}
    \le
    \norm{\ocoEstModel_t}\norm{\at[]}
    \le
    3\ocoModelDiam \ocoDiam
    +
    {\ocoMaxNoise \dout} \sqrt{8 \log \brk*{\frac{4T}{\delta}}}
    =
    \ocoTotDiam.
\end{align*}
Now, for all $1 \le t \le T$ we use a standard uniform convergence argument (\cref{lemma:ocoUniformConvergence}) with $R = \ocoTotDiam$, and $\delta / 2 T$ to get that with probability at least $1-\delta/2$ simultaneously for all $1 \le t \le T$
\begin{align*}
    \LLofu(\at) &- \LLofu(\at[])
    =
    \LL(\at) - \LL(\at[])
    -
    \pOCOoptimism \brk*{\norm{V_t^{-1/2} \at}_\infty - \norm{V_t^{-1/2} \at[]}_\infty}
    \\
    &
    \le
    \frac{1}{t-1}\sum_{s=1}^{t-1} \brk{\lt[s](\ocoEstModel_t \at) - \lt[s](\ocoEstModel_t \at[])}
    -
    \pOCOoptimism \brk*{\norm{V_t^{-1/2} \at}_\infty - \norm{V_t^{-1/2} \at[]}_\infty}
    +
    2\ocoCostVar \sqrt{\frac{\dout \log \frac{6 T^2}{\ocoCostVar \delta}}{t-1}}
    \\
    &
    \le
    2\ocoCostVar \sqrt{\frac{\dout \log \frac{6 T^2}{\ocoCostVar \delta}}{t-1}},
\end{align*}
where the last inequality is by definition of $\at$ as the optimistic cost minimizer. Finally, notice that $\sum_{t=2}^{T} (t-1)^{-1/2} \le 2 \sqrt{T}$ to get that
\begin{align*}
    \ocoR{2}
    =
    \sum_{t=2}^{T} \LLofu(\at) - \LLofu(\at[])
    \le
    \sum_{t=2}^{T} 2\ocoCostVar \sqrt{\frac{\dout \log \frac{6T^2}{\delta}}{t-1}}
    \le
    6\ocoCostVar \sqrt{T \dout \log \frac{3 T}{\ocoCostVar \delta}}
    .
\end{align*}
Finally, taking a union bound on both events and
substituting for the chosen value of $\alpha$ completes the proof.
\end{proof}

\subsection{Deferred details}
\label{sec:ocoAdditionalDetails}

Here we complete the deferred details in the proof of \cref{thm:scoRegret}.
We start with the following high-probability error bound for least squares estimation, that bounds the error of our estimates $\ocoEstModel_t$ of $\ocoTrueModel$, and as such also satisfies the condition of \cref{lemma:ocoOptimismBound}.
\begin{lemma}[\citealp{abbasi2011regret}] 
\label{lemma:ocoParameterEst}
Let $\Delta_t = \ocoTrueModel - \ocoEstModel_t$, and suppose that $\norm{\at}^2 \le \lambda = \ocoDiam^2$, $T \ge \dout$. 
With probability at least $1 - \delta$, we have for all $t \ge 1$
\begin{align*}
    \norm{\Delta_t}_{V_t}^2
    \le
    \tr{\Delta_t\tran V_t \Delta_t}
    \le
    8 \ocoMaxNoise^2 \dout^2 \log \frac{T}{\delta}
    +
    2\ocoDiam^2 \ocoModelDiam^2
    \le
    \frac{\pOCOoptimism^2}{\din}
    .
\end{align*}
\end{lemma}

Next, is a well-known bound on harmonic sums \citep[see, e.g.,][]{cohen2019learning}. This is used to show that the optimistic and true losses are close on the realized predictions (proof in \cref{sec:technicalProofs}).
\begin{lemma}
    \label{lemma:harmonicBound}
    Let $\at \in \RR[\din]$ be a sequence such that $\norm{\at}^2 \le \lambda$, and define $V_t = \lambda I + \sum_{s=1}^{t-1} \at[s] \at[s]\tran$. Then
    $
        \sum_{t=1}^{T} \at\tran V_t^{-1} \at
        \le
        5 \din \log T
        .
    $
\end{lemma}

Finally, a standard uniform convergence result, which is used in bounding $\ocoR{2}$ (proof in \cref{sec:technicalProofs}).
\begin{lemma}
\label{lemma:ocoUniformConvergence}
    Let $R > 0$ and suppose that 
    $
    T
    \ge
    \max\brk[c]{\ocoCostVar, 64 R^2}
    .
    $
    Then for any $\delta \in (0, 1)$ we have that with probability at least $1 - \delta$
    \begin{align*}
        \abs*{
        \sum_{t = 1}^{T}
        \lt(\qq)
        -
        \LL(\qq)
        }
        \le
        \ocoCostVar \sqrt{T \dout \log \frac{3 T}{\ocoCostVar \delta}}
        ,
        \quad
        \forall \qq \in \RR[\dout]
        \;
        \text{s.t.}
        \;
        \norm{\qq} \le R.
    \end{align*}
\end{lemma}

\subsection{Extension to the general control setting}

    Using the assumption that $\Astar$ is strongly stable, we show it takes $2H$ time steps for one of our DAP policies to sufficiently approximate its steady-state. 
    As a result the system may behave arbitrarily during the first $2H$ steps of each subepoch, and we bound the instantaneous regret in each of these time steps by a worst-case constant. 
    As mentioned, though, we show that the total number of subepochs is at most $4(\dx+\du)H \log^2 T$, making the cumulative regret during changes of policy negligible for reasonably large $T$. 
    
    Concretely, let $\RxuMax \ge \max_{\pi \in \Pi_\mathrm{DAP} \cup \pi_{\mathrm{alg}}, t \le T} \max \brk[c]{\norm{(\xtpi, \utpi)}, \norm{(\xt, \ut)}}$ be a bound on the state action magnitude. Combining with the Lipschitz assumption, we get that
    \begin{align*}
        \abs*{
        c_t(x_t, u_t)
        -
        c_t(\xtpi, \utpi)
        }
        \le
        2 \RxuMax
        .
    \end{align*}
    Since the first two sub epochs are always at most $2H$ long, the regret decomposes as
    \begin{align*}
        \mathrm{Regret}_T(\pi)
        \le
        16 \RxuMax H^2 (\dx + \du) \log^2 T
        +
        \sum_{i=1}^{\nEpochs} \sum_{j=3}^{\nSubEpochs}
        \sum_{t = \tauIJ + H}^{\tauIJ[j+1]-1} c_t(\xt,\ut)
        - c_t(\xtpi,\utpi)
        ,
    \end{align*}
    where $\nEpochs$ is the number of epochs and $\nSubEpochs$ is the number of subepochs in epoch $i$.
    
    We proceed by decomposing the remaining term, analyzing the regret within each subepoch.
    For this purpose, define an expected surrogate cost and its optimistic version
    \begin{align*}
        f_t(M; \model, \ww, \zt[])
        &=
        c_t(x_t(M; \model, \ww), u_t(M; \ww), \zt[])
        \\
        F(M; \model)
        &
        =
        \EE[{\zt[]}, \ww] f_t(M; \model, \ww, \zt[])
        \\
        F(M)
        &
        =
        F(M; \model_{\star})
        \\
        \bar{F}_t(M)
        &
        =
        F(M; \model_{\tauIt})
        -
        \pOptimism \maxNoise \norm{V_{\tauIt}^{-1/2} \obsOp(M)}_\infty
        ,
    \end{align*}
    where $i(t) = \max\brk[c]{i \;:\; \tauI \le t}$ is the index of the epoch to which $t$ belongs, and $\norm{\cdot}_\infty$ is the entry-wise matrix infinity norm.
    Letting $M_\star \in \mathcal{M}$ be the DAP approximation of $\pi \in \Pi_\text{lin}$, we have the following decomposition of the instantaneous regret:
    \begin{align*}
        c_t(\xt,\ut) - c_t(\xtpi,\utpi)
        \tag{$\ocoR{1}$ - Truncation + Concentration}
        &
        =
        c_t(\xt,\ut) - F(M_{t})
        \\
        \tag{$\ocoR{2}$ - Optimism}
        &
        +
        F(M_{t}) - \bar{F}_t(M_{t})
        \\
        \tag{$\ocoR{3}$ - Excess Risk}
        &
        +
        \bar{F}_t(M_{t}) - \bar{F}_t(M_\star)
        \\
        \tag{$\ocoR{4}$ - Optimism}
        &
        +
        \bar{F}_t(M_\star) - F(M_\star)
        \\
        \tag{$\ocoR{5}$ - Truncation + Concentration}
        &
        +
        F(M_\star) - c_t(\xtpi,\utpi)
        .
    \end{align*}
    
    \looseness=-1
    The proof is completed by bounding each of the terms above with high probability and combining them with a union bound.
    Terms $\ocoR{2}, \ocoR{3}, \ocoR{4}$ are similar to the ones appearing in the proof \cref{thm:scoRegret}, while terms $\ocoR{1}, \ocoR{5}$ are new and relate the cost to that of the unrolled transition model.
    The latter two terms are bounded in two steps. We start by relating $c_t(x_t, u_t), c_t(\xtpi,\utpi)$ to $f_t(M_t; \model_\star, \ww, \zt), f_t(M_\star, \model_\star, \ww, \zt)$, which, similarly to \cite{agarwal2019online}, uses the assumption that $\Astar$ is strongly stable, but then, also accounts for the discrepancies between $w_t$ and $\hat w_t$ as was done in \cite{NEURIPS2020_565e8a41}. We then conclude with a concentration argument that relates $f_t(M)$ to $F(M)$, which is its expectation with respect to $\ww, \zt[]$. 
    
    To bound the optimism-related terms $\ocoR{2}, \ocoR{4}$, we first show a least-squares confidence bound similar to \cref{lemma:ocoParameterEst}.
    Notice that the least squares bound has to handle the fact that we use the estimated noises $\hat w_t$ to predict the parameters of the unrolled model. 
    Denoting 
    $
        \Delta_t
        =
        \model_\star - \model_t
        ,
    $
    we specifically show that:
    \begin{align*}
        \norm{\Delta_t}_{V_t}^2
        \le
        \tr{\Delta_t\tran V_t \Delta_t}
        \le
        16 \maxNoise^2 \dx^2 \log \brk*{\frac{ T}{\delta}}
        +
        4 \lambda_{\model} \norm{\model_\star}_F^2
        +
        2{\sum_{s=1}^{t-1}\norm{e_s}^2},
    \end{align*}
    a comparable bound to that of \cref{lemma:ocoParameterEst} except for the addition of error terms $e_s$.
    Since $\hat w_t$ converge to the true noises $w_t$, we can prove that 
    $
        \sum_{t=1}^{T} \norm{e_t}^2 \lessapprox \log T
        ,
    $
    i.e., the additional error terms are of a similar order to the standard estimation error and thus do not increase it significantly.
    Next, assuming that the above estimation error holds, we show an analogous result to \cref{lemma:ocoOptimismBound} stating 
    \begin{align*}
        0
        \le
        F(M)
        -
        \bar{F}_t(M)
        \le
        2\pOptimism \maxNoise \norm{V_{\tauIt}^{-1/2} \obsOp(M)}_F
        ,
    \end{align*}
    which yields that $\ocoR{4} \le 0$. To Bound $\ocoR{2}$, we further relate $\norm{V_{\tauIt}^{-1/2} \obsOp(M)}_F$ to $\norm{V_{t}^{-1/2} \obs_{t-1}(M; \wwhat)}$, which is the equivalent of the harmonic term in the right hand side of \cref{lemma:ocoOptimismBound}. To that end, define the noise covariance $\Sigma = \EE w_{t-2H:t-2} w_{t-2H:t-2}\tran$, and notice that our minimum eigenvalue assumption implies that $\norm{\Sigma^{-1/2}} \le \wMin^{-1}$. We thus have that
    \begin{align*}
        \wMin^{2}\norm{V_{\tauIt}^{-1/2} \obsOp(M)}_F^2
        \le
        \norm{V_{\tauIt}^{-1/2} \obsOp(M) \Sigma^{1/2}}_F^2
        &
        =
        \tr{V_{\tauIt}^{-1} \obsOp(M) \Sigma \obsOp(M)}
        \\
        &
        =
        \tr{V_{\tauIt}^{-1} \obsOp(M) \EE\brk[s]{w_{t-2H:t-2} w_{t-2H:t-2}\tran} \obsOp(M)\tran}
        \\
        \tag{\cref{eq:bounded-mem-rep}}
        &
        =
        \EE[\ww]
        \tr{V_{\tauIt}^{-1} \obs_{t-1}(M; \ww) \obs_{t-1}(M; \ww)\tran}
        \\
        &
        =
        \EE[\ww]\norm{V_{\tauIt}^{-1/2} \obs_{t-1}(M; \ww)}^2
        .
    \end{align*}
    Summing over $t$ and applying several technical concentration and noise estimation arguments yields the desired term and the $O(\sqrt{T})$ bound on $\ocoR{2}$.
    
    Last, we deal with $\ocoR{3}$, which is analogous the excess risk term in the proof of \cref{thm:scoRegret}.
    We first show a uniform convergence property akin to \cref{lemma:ocoUniformConvergence}.
    Here, however, the uniform convergence is done with respect to both the randomness in the loss function $\zt[t]$ and the noise terms $w_t$.
    This allows us to use observations gathered in previous subepochs to estimate the expected performance of a DAP policy in the current subepoch.
    Here, we once again tackle the technical difficulty that our DAP policy is defined with respect to the noise estimates $\hat w_t$. 
    This adds an additional error term, proportional to the error in the noise estimates, and accumulates to $\approx \sqrt{T}$ regret overall.

\begin{acks}
This work was partially supported by the Deutsch Foundation, by the Israeli Science Foundation (ISF) grant 2549/19, by the Len Blavatnik and the Blavatnik Family foundation, by the Yandex Initiative in Machine Learning, and by the Israeli VATAT data science scholarship.
\end{acks}

\bibliography{bibliography}

\appendix

\section{Proof of \cref{thm:lqrRegret}}
\label{sec:proofOfLqrRegret}

\begin{theorem}[restatement of \cref{thm:lqrRegret}]
\label{thm:lqrRegretFull}
    Let $\delta \in (0,1)$ and suppose that we run \cref{alg:lqr} with parameters $\RM, \Bbound \ge 1$ and
    \begin{align*}
        H 
        =
        \gamma^{-1} \log T
        ,
        \quad
        \pRegW
        =
        5 \kappa^2 \maxNoise^2 \RM^2 \Bbound^2 H \gamma^{-1}
        ,
        \quad
        \pRegTheta
        =
        2 \maxNoise^2 \RM^2 H^2
        ,
        \\
        \pOptimism
        =
        30 \maxNoise \RM \Bbound \kappa^2 (\dx + \du) H^2 
        \sqrt{\dx \gamma^{-3} (\dx^2 \kappa^2 + \du \Bbound^2) \log \frac{12T}{\delta}}
        .
    \end{align*}
    If 
    $
    T
    \ge
    64 \RM^2
    $
    then with probability at least $1-\delta$ simultaneously for all $\pi \in \Pi_{\mathrm{DAP}}$
    \begin{align*}
        \mathrm{regret}_T(\pi)
        &
        \le
        2860  
        \kappa^3 \gamma^{-11/2} \wMin^{-1} \maxNoise^2 \RM^2 \Bbound^2 
        (\dx^2 \kappa^2 + \du \Bbound^2)
        \sqrt{
        T \dx (\dx + \du)^3 
        }
        \log^5 \frac{18T^5}{\delta}
        \\
        &
        +
        12\costVar
        \sqrt{T \gamma^{-3} (\dx + \du) (\dx^2 \kappa^2 + \du^2 \Bbound^2)}
        \log^3 \frac{18 T^5}{\delta}
\end{align*}
\end{theorem}

\paragraph{Structure.}
We begin with a preliminaries section (\cref{sec:lqrProofPrelims}) that states several results that will be used throughout, and are technical adaptations of existing results.
Next, in \cref{sec:lqrRegretDecomposition} we provide the body of the proof, decomposing the regret into logical terms, and stating the bound for each one.
Finally, in \cref{sec:truncationCostProof,sec:optimismCostProof,sec:excessRiskCostProof} we prove the bounds for each term.

\subsection{Preliminaries}
\label{sec:lqrProofPrelims}

\paragraph{Disturbance estimation.}
The success of our algorithm relies on the estimation of the system disturbances.
The following result, due to \cite{NEURIPS2020_565e8a41}, bounds this estimation error (see proof in \cref{sec:technicalProofs} for completeness).
\begin{lemma}
\label{lemma:disturbanceEstimation}
    Suppose that $\pRegW = 5 \kappa^2 \maxNoise^2 \RM^2 \Bbound^2 H \gamma^{-1}$, $H > \log T$, and $T \ge \dx$. With probability at least $1 - \delta$
    \begin{align*}
        \sqrt{\sum_{t=1}^{T} \norm{w_t - \hat{w}_t}^2}
        \le
        \wErr,
        \quad 
        \text{where}
        \quad 
        \wErr 
        = 
        10 \maxNoise \kappa \RM \Bbound \gamma^{-1} \sqrt{H (\dx + \du) (\dx^2 \kappa^2 + \du \Bbound^2) \log \frac{T}{\delta}}
        .
    \end{align*}
\end{lemma}
As noted by \cite{NEURIPS2020_565e8a41}, the quality of the disturbance estimation does not depend on the choices of the algorithm, i.e., we can recover the noise without any need for exploration. This is in stark contrast to the estimation of the system matrices $\Astar, \Bstar$, which requires exploration.

\paragraph{Estimating the unrolled model.}
Recall from \cref{eq:lqrUnrolling} that by recursing over the transition model we get that
\begin{align*}
    x_t
    =
    \model_\star \tilde{\obs}_{t-1} + {w}_{t-1} + \Astar^{H} x_{t-H}
    ,
\end{align*}
where $\tilde{\obs}_{t-1} = \brk[s]{u_{t-H}\tran, \ldots, u_{t-1}\tran, w_{t-H}\tran, \ldots, w_{t-2}\tran}$. Since we assume $\Astar$ to be strongly stable, the last term is negligible and we essentially recover the standard setting for least squares estimation.
However, notice that $\tilde{\obs}_{t-1}$ cannot be computed as it requires exact knowledge of the disturbances. We thus define its proxy
$
\obs_{t-1}
=
\brk[s]{u_{t-H}\tran, \ldots, u_{t-1}\tran, \hat{w}_{t-H}\tran, \ldots, \hat{w}_{t-2}\tran}
,
$
which replaces $w$ with $\hat{w}$. rewriting the above, we get that
\begin{align*}
    x_t
    =
    \model_\star \obs_{t-1} + {w}_{t-1} + e_{t-1}
    ,
\end{align*}
where
$
    e_{t-1}
    =
    \Astar^{H} x_{t-H}
    +
    \sum_{h=1}^{H} \Astar^{h-1} ({w}_{t-h} - \hat{w}_{t-h})
$
is a bias term, which, while small, is not negligible.
The following result takes the least squares estimation error bound of \citealp{abbasi2011regret}, and augments it with a sensitivity analysis with respect to the observations (see proof in \cref{sec:technicalProofs}).

\begin{lemma} \label{lemma:lqrParameterEst}
Let 
$
\Delta_t
=
\model_\star - \model_t
,
$
and suppose that $\norm{\obs_t}^2 \le \pRegTheta, T \ge \dx$.
With probability at least $1 - \delta$, we have for all $1 \le t \le T$
\begin{align*}
    \norm{\Delta_t}_{V_t}^2
    \le
    \tr{\Delta_t\tran V_t \Delta_t}
    \le
    16 \maxNoise^2 \dx^2 \log \brk*{\frac{ T}{\delta}}
    +
    4\pRegTheta \norm{\model_\star}_F^2
    +
    2{\sum_{s=1}^{t-1}\norm{e_s}^2}.
\end{align*}
If we also have that
$
\pRegTheta
=
2 \maxNoise^2 \RM^2 H^2
,
$
and that 
$
\sum_{t=1}^{T} \norm{w_t - \hat{w}_t}^2
\le
\wErr^2
$
(see \cref{lemma:disturbanceEstimation})
then
\begin{align*}
    \norm{\Delta_t}_{V_t}
    \le
    \sqrt{\tr{\Delta_t\tran V_t \Delta_t}}
    \le
    21 \maxNoise \RM \Bbound \kappa^2 H 
    \sqrt{\gamma^{-3} (\dx + \du) (\dx^2 \kappa^2 + \du \Bbound^2) \log \frac{T}{\delta}}
    ,
\end{align*}
and
$
    \norm{\brk{\model_{t} \; I}}_F
    \le
    17 \Bbound \kappa^2 
    \sqrt{\gamma^{-3} (\dx + \du) (\dx^2 \kappa^2 + \du \Bbound^2) \log \frac{T}{\delta}}
    .
$
\end{lemma}

\paragraph{DAP bounds and properties.}
We need several properties that relate to the DAP parameterization and will be useful throughout.
To that end, we have the following lemma (see proof in \cref{sec:technicalProofs}).
\begin{lemma}
\label{lemma:technicalParameters}
    We have that for all $\ww$ such that $\norm{w_t} \le \maxNoise$, $M \in \mathcal{M}$, and $t \le T$
    \begin{enumerate}
        \item 
        $
        \norm{\brk*{\model_\star \; I}}_F 
        \le
        \Bbound \kappa \sqrt{2 \dx / \gamma}
        ;
        $
        
        \item 
        $
        \norm{u_t(M; \ww)} 
        \le
        \maxNoise \RM \sqrt{H}
        ;
        $
        
        \item 
        $
        \norm{(\obs_t(M; \ww)\tran \; w_{t}\tran)}
        \le
        \sqrt{2} \maxNoise \RM H
        ;
        $
        
        \item 
        $
        \max\brk[c]{\norm{x_t(M; \model_\star; \ww)}, \norm{x_t^{\pi_M}}, \norm{x_t}}
        \le
        2 \kappa \Bbound \maxNoise \RM \sqrt{H} / \gamma
        ;
        $
        
        \item 
        $
        \norm{u_t(M; \ww) - u_t(M; \ww')}
        \le
        \RM \norm{w_{t-H:t-1} - w_{t-H:t-1}'}
        ;
        $
        
        \item 
        $
        \sqrt{\norm{\obs_t(M; \ww) - \obs_t(M; \ww')}^2
        +
        \norm{w_t - w_t'}^2}
        \le
        \RM \sqrt{H} \norm{w_{t-2H:t-1} - w_{t-2H:t-1}'}
        .
        $
    \end{enumerate}
\end{lemma}

\paragraph{Surrogate and optimistic costs.}
We summarize the useful properties of the surrogate costs.
To that end,
with some abuse of notation, we extend the definition of the surrogate and optimistic cost functions to include the dependence on their various parameters:
\begin{equation}
\label{eq:ft-gen-defs}
\begin{aligned}
        f_t(M; \model, \ww, \zt[])
        &=
        c(x_t(M; \model, \ww), u_t(M; \ww); \zt[])
        \\
        \bar{f}_t(M; \model, V, \ww, \zt[])
        &=
        c_t(x_t(M; \model, \ww), u_t(M; \ww); \zt[])
        -
        \pOptimism \maxNoise \norm{V^{-1/2} \obsOp(M)}_{\infty}
        ,
\end{aligned}
\end{equation}
Recalling that $\ww, \wwhat$ are the real and estimated noise sequences respectively, we use the following shorthand notations throughout:
\begin{equation}
\label{eq:ft-short-defs}
\begin{aligned}
    f_t(M)
    &=
    f_t(M; \model_{\star}, \ww, \zt)
    \\
    \bar{f}_t(M)
    &=
    \bar{f}_t(M; \model_{\tauIt}, V_{\tauIt}, \ww, \zt)
    \\
    \hat{f}_t(M)
    &=
    \bar{f}_t(M; \model_{\tauIt}, V_{\tauIt}, \wwhat, \zt)
    ,
\end{aligned}
\end{equation}
where
$
i(t)
=
\max\brk[c]{i : \tauI \le t}
.
$
The following lemma characterizes the properties of $f_t, \bar{f}_t$ as a function of the various parameters (see proof in \cref{sec:technicalProofs}).
\begin{lemma}
\label{lemma:fbarProperties}
    Define the functions
    \begin{align*}
        \maxF(\model) 
        =
        \max\brk[c]*{
        1
        ,
        2\costVar
        +
        5 \RM \maxNoise H \norm{(\model \; I)}
        }
        ,
        \qquad
        \pLipW(\model)
        =
        \sqrt{3H} \RM \norm{(\model \; I)}
        .
    \end{align*}
    For any $\zt[], \zt[]',\ww, \ww'$ with $\norm{w_t}, \norm{w'_t} \le \maxNoise$ and $M, M'$ with $\norm{M}_F, \norm{M'}_F \le \RM$,
    we have:
    \begin{enumerate}
        \item 
        $
        \abs{f_t(M; \model_\star, \ww, \zt[]) - f_t(M; \model_\star, \ww', \zt[]')}
        \le
        2\costVar
        +
        5 \kappa \gamma^{-1} \Bbound \RM \maxNoise \sqrt{H}
        =
        \maxF^{\star}
        ;
        $
        
        \item
        $
        \abs{
        \bar{f}_t(M; \model, V, \ww, \zt[]) 
        -
        \bar{f}_t(M; \model, V, \ww', \zt[]')
        }
        \le
        \maxF(\model)
        $;
        \item
        $
        \abs{
        \bar{f}_t(M; \model, V, \ww)
        -
        \bar{f}_t(M; \model, V, \ww')
        }
        \le
        \pLipW(\model)
        \norm{w_{t-2H:t-1} - w_{t-2H:t-1}'}_F
        ;
        $
    \end{enumerate}
    Additionally, if
    $
    \norm{\brk{\model \; I}}_F
    \le
    17 \Bbound \kappa^2 
    \sqrt{\gamma^{-3} (\dx + \du) (\dx^2 \kappa^2 + \du \Bbound^2) \log \frac{12 T}{\delta}}
    ,
    $
    then:
    \begin{align*}
        \maxF(\model)
        \le
        2\costVar
        +
        3 \pOptimism / (H \sqrt{\dx(\dx+\du)})
        ,
        \quad
        \text{and}
        \;\;
        \pLipW(\model)
        \le
        \pOptimism / (\maxNoise \sqrt{\dx (\dx + \du) H^3})
        .
    \end{align*}
\end{lemma}

\subsection{Regret Decomposition}
\label{sec:lqrRegretDecomposition}

As seen in \cref{eq:lqrUnrolling}, the bounded state representation is such that it depends on the last $H$ decisions of the algorithm. In general, this could be analyzed as an online convex optimization with memory problem, which requires that the cumulative change in predictions be small. However, we use an epoch schedule that ensures a poly-logarithmic number of prediction changes (low switching). This  implies that outside of the first $2H$ rounds in each sub-epoch, the $2H$ step history of each round is fixed, essentially making for a problem with no memory.

The following technical lemma bounds $\nEpochs$, the number of epochs, and $\nSubEpochs$, the number of sub-epochs in epoch $i$ (see proof in \cref{sec:lqrProofs}).
\begin{lemma}
\label{lemma:epochLengths}
We have that $\nEpochs \le 2(\dx+\du)H \log T$ and $\nSubEpochs \le 2 \log T$.
\end{lemma}
We are now ready to prove \cref{thm:lqrRegretFull}.
\begin{proof}[of \cref{thm:lqrRegretFull}]
The first $2H$ rounds of each sub-epoch are the time it takes the system to reach steady-state after a prediction change. During these short mixing periods it will suffice to bound the regret by a constant. Indeed, by \cref{lemma:technicalParameters} we have that for
$
\RxuMax
=
\kappa \Bbound \maxNoise \RM \sqrt{5H} / \gamma
$
\begin{align*}
    \max_{M \in \mathcal{M}, \norm{w} \le \maxNoise, t \le T} 
    \max\brk[c]{
    \norm{(x_t, u_t)}
    ,
    \norm{(x_t^{\pi_M}, u_t^{\pi_M})}
    }
    \le
    \RxuMax
    .
\end{align*}
Combining with the Lipschitz assumption, we get that
$
    \abs*{
    c_t(x_t, u_t)
    -
    c_t(\xtpi, \utpi)
    }
    \le
    2 \RxuMax
    .
$
Using \cref{lemma:epochLengths} we can thus bound the regret as
\begin{align*}
    \mathrm{Regret}_T(\pi)
    \le
    16 \RxuMax (\dx + \du) H^2 \log^2 T
    +
    \sum_{i=1}^{\nEpochs}
    \sum_{j=3}^{\nSubEpochs}
    \sum_{t=\tauIJ+2H}^{\tauIJ[j+1]-1}
    c_t(\xt,\ut) - c_t(\xtpi, \utpi)
    ,
\end{align*}
where we can start with $j \ge 3$ since \cref{alg:lqr} ensures that the first two sub epoch are $2H$ long.
Now, recall the definitions of the expected surrogate and optimistic cost functions, which rely on the bounded memory representations in \cref{eq:bounded-mem-rep}.
\begin{align*}
    f_t(M; \model, \ww, \zt[])
    &=
    c_t(x_t(M; \model, \ww), u_t(M; \ww), \zt[])
    \\
    F(M; \model)
    &
    =
    \EE[{\zt[]}, \ww] f_t(M; \model, \ww, \zt[])
    \\
    F(M)
    &
    =
    F(M; \model_{\star})
    \\
    \bar{F}_t(M)
    &
    =
    F(M; \model_{\tauIt})
    -
    \pOptimism \maxNoise \norm{V_{\tauIt}^{-1/2} \obsOp(M)}_\infty
    ,
\end{align*}
where $i(t) = \max\brk[c]{i \;:\; \tauI \le t}$ is the index of the epoch to which $t$ belongs, and $\norm{\cdot}_\infty$ is the entry-wise matrix infinity norm.
Letting $M_\star \in \mathcal{M}$ be the DAP approximation of $\pi \in \Pi_\text{lin}$, we have the following decomposition of the remaining regret term:
\begin{align*}
    \sum_{i=1}^{\nEpochs}
    \sum_{j=3}^{\nSubEpochs}
    \sum_{t=\tauIJ+2H}^{\tauIJ[j+1]-1}
    &
    c_t(\xt,\ut) - c_t(\xtpi,\utpi)
    \\
    \tag{$\ocoR{1}$ - Truncation + Concentration}
    &
    =
    \sum_{i=1}^{\nEpochs}
    \sum_{j=3}^{\nSubEpochs}
    \sum_{t=\tauIJ+2H}^{\tauIJ[j+1]-1}
    c_t(\xt,\ut) - F(M_{\tauIJ})
    \\
    \tag{$\ocoR{2}$ - Optimism}
    &
    +
    \sum_{i=1}^{\nEpochs}
    \sum_{j=3}^{\nSubEpochs}
    \sum_{t=\tauIJ+2H}^{\tauIJ[j+1]-1}
    F(M_{\tauIJ}) - \bar{F}_t(M_{\tauIJ})
    \\
    \tag{$\ocoR{3}$ - Excess Risk}
    &
    +
    \sum_{i=1}^{\nEpochs}
    \sum_{j=3}^{\nSubEpochs}
    \sum_{t=\tauIJ+2H}^{\tauIJ[j+1]-1}
    \bar{F}_t(M_{\tauIJ}) - \bar{F}_t(M_\star)
    \\
    \tag{$\ocoR{4}$ - Optimism}
    &
    +
    \sum_{i=1}^{\nEpochs}
    \sum_{j=3}^{\nSubEpochs}
    \sum_{t=\tauIJ+2H}^{\tauIJ[j+1]-1}
    \bar{F}_t(M_\star) - F(M_\star)
    \\
    \tag{$\ocoR{5}$ - Truncation + Concentration}
    &
    +
    \sum_{i=1}^{\nEpochs}
    \sum_{j=3}^{\nSubEpochs}
    \sum_{t=\tauIJ+2H}^{\tauIJ[j+1]-1}
    F(M_\star) - c_t(\xtpi,\utpi)
    .
\end{align*}
The proof of \cref{thm:lqrRegret} is concluded by taking a union bound over the following lemmas, which bound each of the terms (see proofs in \cref{sec:truncationCostProof,sec:optimismCostProof,sec:excessRiskCostProof}). The technical derivation of the final regret bound is purely algebraic and may be found in \cref{lemma:regret-final-bound}.
\end{proof}

\begin{lemma}[Truncation cost]
\label{lemma:lqrR15}
    With probability at least $1 - \delta / 3$ we have that
    \begin{align*}
    \ocoR{1} + \ocoR{5}
    \le
    (2\costVar
    +
    7 \kappa^2 \gamma^{-2} \Bbound^2 \RM^2 \maxNoise \sqrt{H})
    \sqrt{32 T H^3 (\dx + \du) (\dx^2 \kappa^2 + \du^2 \Bbound^2) \log^3 \frac{18 T^4}{\delta}}
    .
\end{align*}
\end{lemma}

\begin{lemma}[Optimism cost]
\label{lemma:lqrR24}
    With probability at least $1 - \delta / 3$ we have that $\ocoR{4} \le 0$ and
    \begin{align*}
    \ocoR{2}
    \le
    20 \pOptimism \kappa \gamma^{-1} \wMin^{-1} \maxNoise \RM \Bbound H
    \sqrt{
    T (\dx + \du) (\dx^2 \kappa^2 + \du \Bbound^2) \log^3 \frac{12T}{\delta}
    }
    .
\end{align*}
\end{lemma}

\begin{lemma}[Excess risk]
\label{lemma:lqrR3}
    With probability at least $1 - \delta / 3$ we have that
    \begin{align*}
    \ocoR{3}
    \le
    104 \pOptimism \kappa \gamma^{-1} \RM \Bbound \sqrt{T H \du (\dx^2 \kappa^2 + \du \Bbound^2) \log^3 \frac{18 T^5}{\delta}}
    +
    16 \costVar \sqrt{
    T
    H^3 \dx \du \log^3 \frac{18 T^5}{\delta}
    }
    .
\end{align*}
\end{lemma}

\subsection{Proof of \cref{lemma:lqrR15}}
\label{sec:truncationCostProof}
There are several contributing sub-terms to $\ocoR{1}$ and $\ocoR{5}$. There are the truncation error due the bounded memory representation, and the error due to the disturbance estimation, both of which are standard in recent literature. Additionally, there is a concentration of measure argument.
To that end, we need the following uniform convergence result for sums of random functions that are independent at $2H$ increments (see proof in \cref{sec:technicalProofs}).
\begin{lemma}[Block uniform convergence]
\label{lemma:blockUC}
Let $f_t : \RR[d] \to \RR$ be a sequence of identically distributed functions such that $f_t$ and $f_{t+2H}$ are interdependently distributed. Let $F(M) = \EE f_t(M)$, $R > 0$ and suppose that $T \ge {64 R^2}$ where $C > 1$ is such that
\begin{align*}
    \abs{f_t(M) - f_{t+2H}(M)}
    \le
    C
    \;\;
    ,
    \forall t \ge 1, \norm{M} \le R
    .
\end{align*}
The with probability at least $1-\delta$
\begin{align*}
    \abs*{\sum_{t=1}^{T} (f_t(M) - F(M))}
    \le
    C \sqrt{2 T H d \log \frac{3 T^2}{\delta}}
    .
\end{align*}
\end{lemma}
We start by defining a so-called good event that will be assumed to hold throughout. Suppose that
\begin{align*}
    \sqrt{\sum_{t=1}^{T} \norm{w_t - \hat{w}_t}^2}
    \le
    \wErr 
    = 
    10 \maxNoise \kappa \RM \Bbound \gamma^{-1} \sqrt{H (\dx + \du) (\dx^2 \kappa^2 + \du \Bbound^2) \log \frac{6T}{\delta}}
    ,
\end{align*}
and
\begin{align*}
    \abs*{
    \sum_{t=t_1}^{t_2-1}
    f_t(M) - F(M)
    }
    \le
    (2\costVar
    +
    5 \kappa \gamma^{-1} \Bbound \RM \maxNoise \sqrt{H})
    H
    \sqrt{2(t_2 - t_1) \dx \du \log \frac{18 T^4}{\delta}}
    ,
\end{align*}
for all $M \in \mathcal{M}$, and $1 \le t_1 < t_2 \le T+1$. Notice that by \cref{lemma:fbarProperties}, we have that \cref{lemma:blockUC} holds with 
$
C 
=
2\costVar
+
5 \kappa \gamma^{-1} \Bbound \RM \maxNoise \sqrt{H}
.
$
We thus take a union bound on \cref{lemma:disturbanceEstimation,lemma:blockUC} to get that the above hold with probability at least $1-\delta/3$.

We start with the simpler $\ocoR{5}$. By \cref{lemma:technicalParameters} we have that for
$
\RxuMax
=
\kappa \Bbound \maxNoise \RM \sqrt{5H} / \gamma
$
\begin{align*}
    \max_{M \in \mathcal{M}, \norm{w} \le \maxNoise, t \le T} 
    \max\brk[c]{
    \norm{(x_t, u_t)}
    ,
    \norm{(x_t^{\pi_M}, u_t^{\pi_M})}
    ,
    \norm{(x_t(M; \ww), u_t(M; \ww))}
    }
    \le
    \RxuMax
    .
\end{align*}
Notice that $\utpi = u_t(M_\star; \ww)$. We can thus use the Lipschitz assumption to get that for all $t \ge 1$
\begin{align*}
    f_t(M_\star) - c_t(\xtpi,\utpi)
    &
    =
    c_t(x_t(M_\star; \ww), u_t(M_\star; \ww))
    -
    c_t(\xtpi,\utpi)
    \\
    &
    \le
    \norm{{x}_t(M_\star; \ww) - \xtpi}
    \\
    \tag{\cref{eq:lqrUnrolling}}
    &
    =
    \norm{\Astar^H x_{t-H}^\pi}
    \\
    \tag{strong stability}
    &
    \le
    \RxuMax \kappa (1-\gamma)^H
    .
\end{align*}
Summing over all sub-epochs and using that $(1-\gamma)^H \le e^{-\gamma H}$ we conclude that
\begin{align*}
    \ocoR{5}
    &
    =
    \sum_{i=1}^{\nEpochs}
    \sum_{j=3}^{\nSubEpochs}
    \sum_{t=\tauIJ+2H}^{\tauIJ[j+1]-1}
    \brk[s]{
    (F(M_\star) - f_t(M_\star)) 
    +
    (f_t(M_\star) - c_t(\xtpi,\utpi)
    }
    \\
    &
    \le
    \kappa^2 \gamma^{-1}\Bbound \maxNoise \RM \sqrt{5H} e^{-\gamma H} T
    +
    (\costVar
    +
    5 \kappa \gamma^{-1} \Bbound \RM \maxNoise \sqrt{H})
    H
    \sqrt{2 T \sum_{i=1}^{\nEpochs} (\nSubEpochs) \dx \du \log \frac{18 T^4}{\delta}}
    \\
    &
    \le
    \kappa^2 \gamma^{-1}\Bbound \maxNoise \RM \sqrt{5H}
    +
    (\costVar
    +
    5 \kappa \gamma^{-1} \Bbound \RM \maxNoise \sqrt{H})
    \sqrt{8 T H^3 \dx \du (\dx + \du) \log^3 \frac{18 T^4}{\delta}}
    \\
    &
    \le
    (2\costVar
    +
    6 \kappa^2 \gamma^{-1} \Bbound \RM \maxNoise \sqrt{H})
    \sqrt{8 T H^3 \dx \du (\dx + \du) \log^3 \frac{18 T^4}{\delta}}
    .
\end{align*}
Moving to $\ocoR{1}$, notice that for $\tauIJ+2H \le t <\tauIJ[j+1]$ where $i \ge 1, j \ge 3$ we have that
$
u_t = u_t(M_{\tauIJ}; \wwhat)
.
$
We thus have that
\begin{align*}
    \norm{x_t - {x}_t(M_{\tauIJ}; \ww)}
    &
    =
    \norm*{
    \Astar^H x_{t-H}
    +
    \sum_{h=1}^{H} \Astar^{h-1}\Bstar \brk*{
    u_{t-h}(M_{\tauIJ}; \wwhat)
    - 
    u_{t-h}(M_{\tauIJ}; \ww)
    }
    }
    \\
    &
    \le
    \norm{\Astar^H} \norm{x_{t-H}}
    +
    \sum_{h=1}^{H} \norm{\Astar^{h-1}}\norm{\Bstar} \norm{
    u_{t-h}(M_{\tauIJ}; \wwhat)
    - 
    u_{t-h}(M_{\tauIJ}; \ww)
    }
    \\
    &
    \le
    \RxuMax \kappa e^{-\gamma H}
    +
    \kappa \Bbound \sum_{h=1}^{H}  (1-\gamma)^{h-1} \norm{
    u_{t-h}(M_{\tauIJ}; \wwhat)
    - 
    u_{t-h}(M_{\tauIJ}; \ww)
    }
    .
\end{align*}
Denoting $\brk[s]{x}_+ = \max\brk[c]{0, x}$ we thus conclude that
\begin{align*}
    &
    \sum_{i=1}^{\nEpochs}
    \sum_{j=3}^{\nSubEpochs}
    \sum_{t=\tauIJ+2H}^{\tauIJ[j+1]-1}
    c_t(\xt,\ut) - f_t(M_{\tauIJ})
    \\
    &
    =
    \sum_{i=1}^{\nEpochs}
    \sum_{j=3}^{\nSubEpochs}
    \sum_{t=\tauIJ+2H}^{\tauIJ[j+1]-1}
    c_t(\xt,\ut) - c_t(x_t(M_{\tauIJ}; \ww), u_t(M_{\tauIJ}; \ww))
    \\
    \tag{Lipschitz}
    &
    \le
    \sum_{i=1}^{\nEpochs}
    \sum_{j=3}^{\nSubEpochs}
    \sum_{t=\tauIJ+2H}^{\tauIJ[j+1]-1}
    (
    \norm{x_t - x_t(M_{\tauIJ}; \ww)}
    +
    \norm{u_t - u_t(M_{\tauIJ}; \ww)}
    )
    \\
    &
    \le
    \kappa \RxuMax e^{-\gamma H} T 
    +
    \kappa \Bbound 
    \sum_{i=1}^{\nEpochs}
    \sum_{j=3}^{\nSubEpochs}
    \sum_{t=\tauIJ+2H}^{\tauIJ[j+1]-1}
    \sum_{h=0}^{H}
    (1-\gamma)^{\brk[s]{h-1}_+} \norm{
    u_{t-h}(M_{\tauIJ}; \wwhat) - u_{t-h}(M_{\tauIJ}; \ww)}
    \\
    &
    \le
    \kappa \RxuMax
    +
    \kappa \Bbound \RM
    \sum_{i=1}^{\nEpochs}
    \sum_{j=3}^{\nSubEpochs}
    \sum_{t=\tauIJ+2H}^{\tauIJ[j+1]-1}
    \sum_{h=0}^{H}
    (1-\gamma)^{\brk[s]{h-1}_+} \norm{
    w_{t-h-H:t-h-1} - \hat{w}_{t-h-H:t-h-1}}
    \\
    \tag{Jensen}
    &
    \le
    \kappa \RxuMax
    +
    \kappa \Bbound \RM
    \sum_{h=0}^{H}
    (1-\gamma)^{\brk[s]{h-1}_+} 
    \sqrt{T \sum_{t=H}^{T} \norm{
    w_{t-h-H:t-h-1} - \hat{w}_{t-h-H:t-h-1}}^2}
    \\
    &
    \le
    \kappa \RxuMax
    +
    2\kappa \gamma^{-1}\Bbound \RM
    \wErr
    \sqrt{T H}
    \\
    &
    \le
    21 \kappa^2 \gamma^{-2}\Bbound^2 \RM^2 \maxNoise H \sqrt{T (\dx + \du) (\dx^2 \kappa^2 + \du \Bbound^2) \log \frac{6T}{\delta}}
    ,
\end{align*}
where the third inequality used that $u_t(M;\ww)$ is $\RM$-Lipschitz in $\ww$ (\cref{lemma:technicalParameters}), the fifth is since each summand appears at most $H$ times in the sum, and the fifth plugged in the vlaues of $\RxuMax$ and $\wErr$.
Next, repeating the steps in the bound of $\ocoR{5}$, we also have that
\begin{align*}
    \sum_{i=1}^{\nEpochs}
    \sum_{j=3}^{\nSubEpochs}
    \sum_{t=\tauIJ+2H}^{\tauIJ[j+1]-1}
    &
    f_t(M_{\tauIJ}) - F(M_{\tauIJ})
    \\
    &
    \le
    (2\costVar
    +
    5 \kappa \gamma^{-1} \Bbound \RM \maxNoise \sqrt{H})
    H
    \sqrt{2 T \sum_{i=1}^{\nEpochs} (\nSubEpochs) \dx \du \log \frac{18 T^4}{\delta}}
    \\
    &
    \le
    (2\costVar
    +
    5 \kappa \gamma^{-1} \Bbound \RM \maxNoise \sqrt{H})
    \sqrt{8 T H^3 \dx \du (\dx + \du) \log^3 \frac{18 T^4}{\delta}}
    .
\end{align*}
Putting both bounds together, we conclude that
\begin{align*}
    \ocoR{1}
    &
    =
    \sum_{i=1}^{\nEpochs}
    \sum_{j=3}^{\nSubEpochs}
    \sum_{t=\tauIJ+2H}^{\tauIJ[j+1]-1}
    \brk[s]{
    (c_t(\xt, \ut) - f_t (M_{\tauIJ}))
    +
    (f_t(M_{\tauIJ}) - F(M_{\tauIJ}))
    }
    \\
    &
    \le
    21 \kappa^2 \gamma^{-2}\Bbound^2 \RM^2 \maxNoise H \sqrt{T (\dx + \du) (\dx^2 \kappa^2 + \du \Bbound^2) \log \frac{6T}{\delta}}
    \\
    &
    +
    (2\costVar
    +
    5 \kappa \gamma^{-1} \Bbound \RM \maxNoise \sqrt{H})
    \sqrt{8 T H^3 \dx \du (\dx + \du) \log^3 \frac{18 T^4}{\delta}}
    \\
    &
    \le
    (2\costVar
    +
    7 \kappa^2 \gamma^{-2} \Bbound^2 \RM^2 \maxNoise \sqrt{H})
    \sqrt{8 T H^3 (\dx + \du) (\dx^2 \kappa^2 + \du^2 \Bbound^2) \log^3 \frac{18 T^4}{\delta}}
    .
    &&\blacksquare
\end{align*}

\subsection{Proof of \cref{lemma:lqrR24}}
\label{sec:optimismCostProof}
We first need the following lemma that deals with the concentration of sums of variables that are independent when they are $2H$ apart in time (proof in \cref{sec:technicalProofs}).

\begin{lemma}[Block Bernstein]
\label{lemma:blockConcentration}
Let $X_t$ be a sequence of random variables adapted to a filtration $\mathcal{F}_t$. If $0 \le X_t \le 1$ then with probability at least $1 - \delta$ simultaneously for all $1 \le t \le T$
\begin{align*}
    \sum_{s=1}^{t} \EE\brk[s]{X_t \mid \mathcal{F}_{t-2H}}
    \le
    2 \sum_{s=1}^{t} (X_s)
    +
    8 H \log \frac{2 T^2}{\delta}
    .
\end{align*}
\end{lemma}

Next, an optimistic cost function should satisfy two properties. On the one hand, it is a global lower bound on the true cost function. On the other, it has a small error on the realized prediction sequence. Both of these properties are established in the following lemma (see proof in \cref{sec:lqrProofs}).
\begin{lemma}[Optimism]
\label{lemma:lqrOptimism}
    Let 
    $
    \Delta_t
    =
    \model_\star - \model_t
    $
    and suppose that the optimism parameter $\pOptimism$ satisfies
    $
    \sqrt{2 \dx (\dx + \du) H^2
    \tr{\Delta_{\tauIt}\tran V_{\tauIt} \Delta_{\tauIt}}}
    \le
    \pOptimism
    .
    $
    Then we have that for all $M \in \mathcal{M}$
    \begin{align*}
        \bar{F}_t(M)
        \le
        F(M)
        \le
        \bar{F}_t(M)
        +
        2\pOptimism \maxNoise \norm{V_{\tauIt}^{-1/2} \obsOp(M)}_F
    \end{align*}
    If additionally $t \ge \tauIt + 2H$ then
    \begin{align*}
    2 \pOptimism \maxNoise \norm{V_{\tauIt}^{-1/2} \obsOp(M)}_F
    \le
    2 \pOptimism \maxNoise \wMin^{-1}
    \sqrt{\EE[\ww]\norm{V_{\tauIt}^{-1/2} \obs_{t-1}(M; \ww)}^2}
\end{align*}
\end{lemma}
We begin by defining the good event, which will be assumed to hold throughout. Let $\mathcal{F}_t = \sigma(w_1, \ldots w_{t-1}, \zt[1], \ldots, \zt[t-1])$ be the filtration generated by all randomness up to (not including) time $t$. Suppose that for all $i \ge 1$ and $j \ge 3$
\begin{align}
\label{eq:confidence-concentration}
    \sum_{t = \tauIJ+2H}^{\tauIJ[j+1]-1}
    \EE\brk[s]{\norm{V_{\tauI}^{-1/2} \obs_{t-1}(M_{\tauIJ}; \ww)}^2 \mid \mathcal{F}_{t-2H}}
    \le
    2\sum_{t = \tauIJ+2H}^{\tauIJ[j+1]-1}
    \norm{V_{\tauI}^{-1/2} \obs_{t-1}(M_{\tauIJ}; \ww)}^2
    +
    8 H \log \frac{12T^3}{\delta}
    ,
\end{align}
and that
\begin{align*}
    \sqrt{\sum_{t=1}^{T} \norm{w_t - \hat{w}_t}^2}
    &
    \le
    \wErr 
    = 
    10 \maxNoise \kappa \RM \Bbound \gamma^{-1} \sqrt{H (\dx + \du) (\dx^2 \kappa^2 + \du \Bbound^2) \log \frac{12T}{\delta}}
    ,
    \\
    \sqrt{\tr{\Delta_t\tran V_t \Delta_t}}
    &
    \le
    21 \maxNoise \RM \Bbound \kappa^2 H 
    \sqrt{\gamma^{-3} (\dx + \du) (\dx^2 \kappa^2 + \du \Bbound^2) \log \frac{12T}{\delta}}
    \qquad
    ,
    \forall t \le T
    .
\end{align*}
By \cref{lemma:technicalParameters}, we have that
\begin{align*}
    \norm{V_{\tauI}^{-1/2} \obs_{t-1}(M_{\tauIJ}; \ww)}^2
    \le
    \pRegTheta^{-1} \norm{\obs_{t-1}(M_{\tauIJ}; \ww)}^2
    \le
    (\sqrt{2} \maxNoise \RM H)^{-1} (\sqrt{2} \maxNoise \RM H)
    =
    1
    .
\end{align*}
We thus use \cref{lemma:blockConcentration} with $\delta / 6T$ and a union bound over the sub-epochs, and \cref{lemma:lqrParameterEst,lemma:disturbanceEstimation} with $\delta / 12$ to conclude that the above events hold with probability at least $1 - \delta / 3$.
Now, notice that
\begin{align*}
    &
    \sqrt{2 \dx (\dx + \du) H^2
    \tr{\Delta_{\tauIt}\tran V_{\tauIt} \Delta_{\tauIt}}}
    \\
    &
    \le
    30 \maxNoise \RM \Bbound \kappa^2 (\dx + \du) H^2 
    \sqrt{\dx \gamma^{-3} (\dx^2 \kappa^2 + \du \Bbound^2) \log \frac{12T}{\delta}}
    =
    \pOptimism
    ,
\end{align*}
and thus we \cref{lemma:lqrOptimism} holds. From its left hand side, we get that $\ocoR{4} \le 0$, and from its right hand side, we also get that
\begin{align*}
    \ocoR{2}
    &
    \le
    2 \pOptimism \maxNoise \wMin^{-1}
    \sum_{i=1}^{\nEpochs}
    \sum_{j=3}^{\nSubEpochs}
    \sum_{t=\tauIJ+2H}^{\tauIJ[j+1]-1}
    \sqrt{\EE[\ww]\norm{V_{\tauI}^{-1/2} \obs_{t-1}(M_{\tauIJ}; \ww)}^2}
    \\
    \tag{$V_{\tauI},M_{\tauIJ}$ $\mathcal{F}_{t-2H}$ measurable}
    &
    \le
    2 \pOptimism \maxNoise \wMin^{-1}
    \sum_{i=1}^{\nEpochs}
    \sum_{j=3}^{\nSubEpochs}
    \sum_{t=\tauIJ+2H}^{\tauIJ[j+1]-1}
    \sqrt{\EE \brk[s]{\norm{V_{\tauI}^{-1/2} \obs_{t-1}(M_{\tauIJ}; \ww)}^2 \mid \mathcal{F}_{t-2H}}}
    \\
    \tag{Jensen}
    &
    \le
    2 \pOptimism \maxNoise \wMin^{-1}
    \sqrt{T \sum_{i=1}^{\nEpochs}
    \sum_{j=3}^{\nSubEpochs}
    \sum_{t=\tauIJ+2H}^{\tauIJ[j+1]-1}
    \EE \brk[s]{\norm{V_{\tauI}^{-1/2} \obs_{t-1}(M_{\tauIJ}; \ww)}^2 \mid \mathcal{F}_{t-2H}}}
    \\
    \tag{\cref{eq:confidence-concentration}}
    &
    \le
    2 \pOptimism \maxNoise \wMin^{-1}
    \sqrt{T \brk[s]*{ \sum_{i=1}^{\nEpochs}
    \sum_{j=3}^{\nSubEpochs}
    8 H \log \frac{12T^3}{\delta}
    +
    2
    \sum_{t=\tauIJ+2H}^{\tauIJ[j+1]-1}
    \norm{V_{\tauI}^{-1/2} \obs_{t-1}(M_{\tauIJ}; \ww)}^2
    }}
    \\
    \tag{\cref{lemma:epochLengths}}
    &
    \le
    2 \pOptimism \maxNoise \wMin^{-1}
    \sqrt{T \brk[s]*{
    32 H^2 (\dx + \du) \log^3 \frac{12T^3}{\delta}
    +
    2
    \sum_{i=1}^{\nEpochs}
    \sum_{j=3}^{\nSubEpochs}
    \sum_{t=\tauIJ+2H}^{\tauIJ[j+1]-1}
    \norm{V_{\tauI}^{-1/2} \obs_{t-1}(M_{\tauIJ}; \ww)}^2
    }}
    .
\end{align*}
We bound the remaining term in two steps. First, we use the fact that $\obs_t(M; \ww)$ is $\RM \sqrt{H}-$Lipschitz with respect to $w_{t-2H:t-1}$ (see \cref{lemma:technicalParameters}) to get that
\begin{align*}
    \sum_{i=1}^{\nEpochs}
    \sum_{j=3}^{\nSubEpochs}
    \sum_{t=\tauIJ+2H}^{\tauIJ[j+1]-1}
    &
    \norm{V_{\tauI}^{-1/2} \brk[s]{\obs_{t-1}(M_{\tauIJ}; \ww) - \obs_{t-1}(M_{\tauIJ}; \wwhat)}}^2
    \\
    &
    \le
    \pRegTheta^{-1}
    \sum_{i=1}^{\nEpochs}
    \sum_{j=3}^{\nSubEpochs}
    \sum_{t=\tauIJ+2H}^{\tauIJ[j+1]-1}
    \norm{\obs_{t-1}(M_{\tauIJ}; \ww) - \obs_{t-1}(M_{\tauIJ}; \wwhat)}^2
    \\
    &
    \le
    \frac{1}{2 H \maxNoise^2}
    \sum_{i=1}^{\nEpochs}
    \sum_{j=3}^{\nSubEpochs}
    \sum_{t=\tauIJ+2H}^{\tauIJ[j+1]-1}
    \norm{w_{t-2H:t-2} - \hat{w}_{t-2H:t-2}}^2
    \\
    &
    \le
    \frac{1}{\maxNoise^2}
    \sum_{t=1}^{T}
    \norm{w_t - \hat{w}_t}^2
    \\
    &
    \le
    100 \kappa^2 \gamma^{-2} \RM^2 \Bbound^2  H (\dx + \du) (\dx^2 \kappa^2 + \du \Bbound^2) \log \frac{12T}{\delta}
\end{align*}
Next, for any $t \in \brk[s]{\tauIJ + 2H, \tauIJ[j+1]-1}$ notice that 
$
\obs_t = \obs_t(M_{\tauIJ}; \wwhat)
,
$
and that \cref{alg:lqr} ensures that $\det(V_{t-1}) \le 2 \det(V_{\tauI})$.
We thus have that
\begin{align*}
    \sum_{i=1}^{\nEpochs}
    \sum_{j=3}^{\nSubEpochs}
    \sum_{t=\tauIJ+2H}^{\tauIJ[j+1]-1}
    \norm{V_{\tauI}^{-1/2} \obs_{t-1}}^2
    &
    =
    \sum_{i=1}^{\nEpochs}
    \sum_{j=3}^{\nSubEpochs}
    \sum_{t=\tauIJ+2H}^{\tauIJ[j+1]-1}
    \obs_{t-1}\tran V_{\tauI}^{-1} \obs_{t-1}
    \\
    \tag{\cref{lemma:vTau-to-vt}}
    &
    \le
    2
    \sum_{t=2}^{T}
    \obs_{t-1}\tran V_{t-1}^{-1} \obs_{t-1}
    \\
    \tag{\cref{lemma:harmonicBound}}
    &
    \le
    10 (\dx + \du) H \log T
    .
\end{align*}
We conclude that
\begin{align*}
    \sum_{i=1}^{\nEpochs}
    \sum_{j=3}^{\nSubEpochs}
    \sum_{t=\tauIJ+2H}^{\tauIJ[j+1]-1}
    \norm{V_{\tauI}^{-1/2} \obs_{t-1}(M_{\tauIJ}; \ww)}^2
    \le
    220 \kappa^2 \gamma^{-2} \RM^2 \Bbound^2  H (\dx + \du) (\dx^2 \kappa^2 + \du \Bbound^2) \log \frac{12T}{\delta}
    ,
\end{align*}
and plugging this into the bound in $\ocoR{2}$ we get that
\begin{align*}
    &
    \ocoR{2}
    \\
    &
    \le
    2 \pOptimism \maxNoise \wMin^{-1}
    \sqrt{T \brk[s]*{
    32 H^2 (\dx + \du) \log^3 \frac{12T^3}{\delta}
    +
    440 \frac{\kappa^2}{\gamma^{2}} \RM^2 \Bbound^2  H (\dx + \du) (\dx^2 \kappa^2 + \du \Bbound^2) \log \frac{12T}{\delta}
    }}
    \\
    &
    \le
    20 \pOptimism \kappa \gamma^{-1} \wMin^{-1} \maxNoise \RM \Bbound H
    \sqrt{
    T (\dx + \du) (\dx^2 \kappa^2 + \du \Bbound^2) \log^3 \frac{12T}{\delta}
    }
    ,
\end{align*}
where the second inequality also used that $H , \log T \ge 2$.
\hfill$\blacksquare$

\subsection{Proof of \cref{lemma:lqrR3}}
\label{sec:excessRiskCostProof}
Recall that $\ocoR{3}$ is a sum over the excess risk of the optimistic cost function, i.e.,
$
\bar{F}_t(M_{\tauIJ}) - \bar{F}_t(M_\star)
.
$
However, \cref{alg:lqr} does not have access to this cost and thus optimizes its empirical version.
Concretely, define the per time step optimistic cost
\begin{align*}
    \bar{f}_t(M; \model, V, \ww, \zt[])
    &
    =
    c_t(x_t(M; \model, \ww), u_t(M; \ww), \zt[])
    -
    \pOptimism \maxNoise \norm{V^{-1/2} \obsOp(M)}_\infty
    \\
    &
    =
    f_t(M; \model, \ww, \zt[])
    -
    \pOptimism \maxNoise \norm{V^{-1/2} \obsOp(M)}_\infty
    ,
\end{align*}
and its instances
\begin{align*}
    \hat{f}_t(M)
    &
    =
    \bar{f}(M; \model_{\tauIt}, V_{\tauIt}, \wwhat, \zt)
    =
    f_t(M; \model_{\tauIt}, \wwhat, \zt)
    -
    \pOptimism \maxNoise \norm{V_{\tauIt}^{-1/2} \obsOp(M)}_\infty
    \\
    \bar{f}_t(M)
    &
    =
    \bar{f}(M; \model_{\tauIt}, V_{\tauIt}, \ww, \zt)
    =
    f_t(M; \model_{\tauIt}, \ww, \zt)
    -
    \pOptimism \maxNoise \norm{V_{\tauIt}^{-1/2} \obsOp(M)}_\infty
    ,
\end{align*}
where recall that $i(t) = \max\brk[c]{i \;:\; \tauI \le t}$ is the index of the episode to which $t$ belongs.
We start by defining a good event. We assume that the following events hold throughout the proof:
\begin{align*}
    \sqrt{\sum_{t=1}^{T} \norm{w_t - \hat{w}_t}^2}
    &
    \le
    \wErr 
    = 
    10 \maxNoise \kappa \RM \Bbound \gamma^{-1} \sqrt{H (\dx + \du) (\dx^2 \kappa^2 + \du \Bbound^2) \log \frac{12T}{\delta}}
    ,
    \\
    \norm{\brk{\model_{t} \; I}}_F
    &
    \le
    17 \Bbound \kappa^2 
    \sqrt{\gamma^{-3} (\dx + \du) (\dx^2 \kappa^2 + \du \Bbound^2) \log \frac{12T}{\delta}}
    \qquad
    ,
    \forall t \le T
    ,
\end{align*}
and for all $i \ge 1, j \ge 3, M \in \mathcal{M}$
\begin{align*}
    \abs*{
    \sum_{t =\tauIJ[j-1]}^{\tauIJ-1}
    \bar{f}_t(M; \model_{\tauI}, V_{\tauI}, \ww)
    -
    \bar{F}(M; \model_{\tauI}, V_{\tauI})
    }
    \le
    \maxF^{\max} H \sqrt{2 (t_2-t_1) \dx \du \log \frac{18 T^5}{\delta}}
    ,
\end{align*}
where $
\maxF^{\max}
=
\max_{1 \le i \le \nEpochs} \maxF(\model_{\tauI})
,
$
and
$\maxF(\model) \ge 1$ is defined in \cref{lemma:fbarProperties} and bounds the summands.
We show that the above event holds with probability at least $1-\delta/3$. The first two parts each hold with probability $1-\delta/12$ (see \cref{lemma:disturbanceEstimation,lemma:lqrParameterEst} with $\delta/12$). As for the last part, notice that for $t \ge \tauI + 2H$
\begin{align*}
    \EE \brk[s]{\bar{f}_t(M; \model_{\tauI}, V_{\tauI}, \ww)
    \mid
    \model_{\tauI}, V_{\tauI}
    }
    =
    \bar{F}(M; \model_{\tauI}, V_{\tauI})
    ,
\end{align*}
and that conditioned on $\model_{\tauI}, V_{\tauI}$ we have that
$\bar{f}_t(M; \model_{\tauI}, V_{\tauI}, \ww)$ are i.i.d at $2H$ increments.
We thus use \cref{lemma:blockUC} with $\delta / 6 T^3$ and a union bound over epoch, sub-epochs, and interval lengths to get that the last term holds with probability at least $1-\delta/6$. Taking another union bound, all three parts holds simultaneously with probability at least $1-\delta/3$.

Then the optimistic cost minimization in \cref{alg:lqr} can be written as
\begin{align*}
    M_{\tauIJ}
    \in
    \argmin_{M \in \mathcal{M}} 
    \sum_{t=\tauIJ[j-1]}^{\tauIJ-1} \hat{f}_t(M)
    ,
\end{align*}
and thus we have that
\begin{align}
\label{eq:optimisticMin}
    \sum_{t=\tauIJ[j-1]}^{\tauIJ-1} \brk*{
    \hat{f}_t(M_{\tauIJ}) - \hat{f}_t(M)
    }
    \le 
    0
    ,
    \;
    \;
    \forall M \in \mathcal{M}
    .
\end{align}

It remains to relate $\hat{f}_t$ and $\bar{F}_t$.
We start by using \cref{lemma:fbarProperties} to get that for all $M \in \mathcal{M}$
\begin{align*}
    \abs{\bar{f}_t(M) - \hat{f}_t(M)}
    =
    \abs{f_t(M; \model_{\tauIt}, \ww) - f_t(M; \model_{\tauIt}, \wwhat)}
    \le
    \pLipW^{\max} \norm{w_{t-2H:t-1} - \hat{w}_{t-2H:t-1}}
    ,
\end{align*}
where $\pLipW^{\max} = \max_{i \ge 1} (\model_{\tauIt})$, and $\pLipW(\model)$ is defined in \cref{lemma:fbarProperties}.
We thus have
\begin{align*}
    \sum_{t=\tauIJ[j-1]}^{\tauIJ-1}\brk[s]*{
    \bar{f}_t(M_{\tauIJ}) - \bar{f}_t(M_\star)
    }
    &
    \le
    \sum_{t=\tauIJ[j-1]}^{\tauIJ-1}\brk[s]*{
    \hat{f}_t(M_{\tauIJ}) - \hat{f}_t(M_\star)
    +
    2 \pLipW^{\max} \norm{\hat{w}_{t-2H:t-1} - w_{t-2H:t-1}}
    }
    \\
    \tag{\cref{eq:optimisticMin}}
    &
    \le
    2 \pLipW^{\max}
    \sum_{t=\tauIJ[j-1]}^{\tauIJ-1} \norm{\hat{w}_{t-2H:t-1} - w_{t-2H:t-1}}
    \\
    \tag{Jensen}
    &
    \le
    2 \pLipW^{\max} \sqrt{
    (\tauIJ - \tauIJ[j-1]) \sum_{t=\tauIJ[j-1]}^{\tauIJ-1} \norm{\hat{w}_{t-2H:t-1} - w_{t-2H:t-1}}^2
    }
    \\
    &
    \le
    \pLipW^{\max} \sqrt{
    8 (\tauIJ - \tauIJ[j-1]) H \sum_{t=1}^{T} \norm{w_t - \hat{w}_t}^2
    }
    \\
    &
    \le
    \pLipW^{\max} \wErr \sqrt{
    8 (\tauIJ - \tauIJ[j-1]) H
    }
    .
\end{align*}
Now, for $i \ge 1$ and $\tauI+2H \le t \le \tauI[i+1]-1$ we have
\begin{align*}
    &
    (\tauIJ - \tauIJ[j-1])
    (
    \bar{F}_t(M_{\tauIJ}) - \bar{F}_t(M_\star)
    )
    =
    (\tauIJ - \tauIJ[j-1])
    (
    \bar{F}(M_{\tauIJ}; \model_{\tauI}, V_{\tauI})
    -
    \bar{F}(M_\star; \model_{\tauI}, V_{\tauI})
    )
    \\
    &
    \le
    \sum_{t=\tauIJ[j-1]}^{\tauIJ-1}(\bar{f}_t(M_{\tauIJ}) - \bar{f}_t(M_\star))
    +
    \maxF^{\max} H \sqrt{8(\tauIJ - \tauIJ[j-1]) \dx \du \log \frac{18 T^5}{\delta}}
    \\
    &
    \le
    \pLipW^{\max} \wErr \sqrt{
    8 (\tauIJ - \tauIJ[j-1]) H
    }
    +
    \maxF^{\max} H \sqrt{8(\tauIJ - \tauIJ[j-1]) \dx \du \log \frac{18 T^5}{\delta}}
    \\
    &
    =
    \brk[s]*{
    \pLipW^{\max} \wErr 
    +
    \maxF^{\max} \sqrt{H \dx \du \log \frac{18 T^5}{\delta}}
    }
    \sqrt{
    8 (\tauIJ - \tauIJ[j-1]) H
    }
    .
\end{align*}
Since $\norm{(\model_t \; I)}$ is bounded, \cref{lemma:fbarProperties} gives us that
\begin{align*}
    \maxF^{\max}
    \le
    2\costVar
    +
    3 \pOptimism / (H \sqrt{\dx(\dx+\du)})
    ,
    \quad
    \text{and}
    \;\;
    \pLipW^{\max}
    \le
    \pOptimism / (\maxNoise \sqrt{\dx (\dx + \du) H^3})
    .
\end{align*}
Finally, notice that each sub-epoch is at most twice as long as its predecessor. We conclude that
\begin{align*}
    \ocoR{3}
    &
    =
    \sum_{i=1}^{\nEpochs}
    \sum_{j=3}^{\nSubEpochs}
    \sum_{t=\tauIJ+2H}^{\tauIJ[j+1]-1}
    \bar{F}_t(M_{\tauIJ}) - \bar{F}_t(M_\star)
    \\
    &
    \le
    \sum_{i=1}^{\nEpochs}
    \sum_{j=3}^{\nSubEpochs}
    \sum_{t=\tauIJ+2H}^{\tauIJ[j+1]-1}
    \brk[s]*{
    \pLipW^{\max} \wErr 
    +
    \maxF^{\max} \sqrt{H \dx \du \log \frac{18 T^5}{\delta}}
    }
    \sqrt{\frac{8 H}{\tauIJ - \tauIJ[j-1]}
    }
    \\
    &
    \le
    4
    \sum_{i=1}^{\nEpochs}
    \sum_{j=3}^{\nSubEpochs}
    \brk[s]*{
    \pLipW^{\max} \wErr 
    +
    \maxF^{\max} \sqrt{H \dx \du \log \frac{18 T^5}{\delta}}
    }
    \sqrt{H (\tauIJ[j+1] - \tauIJ)}
    \\
    \tag{Jensen}
    &
    \le
    8
    \brk[s]*{
    \pLipW^{\max} \wErr 
    +
    \maxF^{\max} \sqrt{H \dx \du \log \frac{18 T^5}{\delta}}
    }
    H \log(T)
    \sqrt{T (\dx + \du)}
    \\
    &
    \le
    104 \pOptimism \kappa \gamma^{-1} \RM \Bbound \sqrt{T H \du (\dx^2 \kappa^2 + \du \Bbound^2) \log^3 \frac{18 T^5}{\delta}}
    +
    16 \costVar \sqrt{
    T
    H^3 \dx \du \log^3 \frac{18 T^5}{\delta}
    }
    \\
    &
    =
    \brk[s]*{
    104 \pOptimism \kappa \gamma^{-1} \RM \Bbound \sqrt{\dx^2 \kappa^2 + \du \Bbound^2}
    +
    16 \costVar H \sqrt{\dx}
    }
    \sqrt{
    T
    H \du \log^3 \frac{18 T^5}{\delta}
    }
    ,
\end{align*}
where the last inequality plugged in the values of $\wErr, \pLipW^{\max}, \maxF^{\max}.$
\hfill$\blacksquare$

\section{Proofs of Side Lemmas}
\label{sec:lqrProofs}

\begin{proof}[of \cref{lemma:epochLengths}]
    The algorithm ensures that 
    \begin{align*}
    \det(V_T)
    \ge
    \det(V_{\tauI[\nEpochs]}) 
    \ge
    2 \det(V_{\tauI[\nEpochs-1]})
    \ldots
    \ge
    2^{\nEpochs-1} \det{V_1}
    ,
    \end{align*}
    and changing sides, and taking the logarithm we conclude that
    \begin{align*}
    \nEpochs
    &
    \le
    1
    +
    \log \brk*{\det(V_{T}) / \det(V)}
    \\
    &
    =
    1
    +
    \log \det(V^{-1/2} V_{T+1} V^{-1/2})
    \\
    \tag{$\det(A) \le \norm{A}^d$}
    &
    \le
    1
    +
    (\dx+\du)H \log \norm{V^{-1/2} V_{T} V^{-1/2}}
    \\
    \tag{triangle inequality}
    &
    \le
    1
    +
    (\dx+\du)H \log \brk*{1 + \frac{1}{\pRegTheta}\sum_{t=1}^{T-1} \norm{\obs_t}^2}
    \\
    &
    \le
    1
    +
    (\dx+\du)H \log T
    \\
    \tag{$T \ge 3$}
    &
    \le
    2(\dx+\du)H \log T
    ,
\end{align*}
where the second to last inequality holds since $\norm{\obs_t}^2 \le \pRegTheta$ by \cref{lemma:technicalParameters}.
Next, for $\nSubEpochs$ we have
\begin{align*}
    T
    \ge
    \tauIJ[\nSubEpochs] - \tauI
    \ge
    2 \tauIJ[\nSubEpochs-1] - \tauI
    \ldots
    \ge
    2^{\nSubEpochs-2} (\tauIJ[2] - \tauI) \indEvent{\nSubEpochs \ge 2}
    =
    2^{\nSubEpochs-2} \indEvent{\nSubEpochs \ge 2}
    .
\end{align*}
We conclude that either $\nSubEpochs = 1 \le \log T$ or
\begin{align*}
    \nSubEpochs
    \le
    2 + \log_2 T
    \le
    2 \log T
    ,
\end{align*}
where the last inequality holds for $T > 20$.
\end{proof}

\begin{proof}
[of \cref{lemma:lqrOptimism}]
Recall that
$
F(M ;\model)
=
\EE[\zt,\ww] c_t(x_t(M; {\model}, \ww), u_t(M; \ww))
$
and let $\Delta = \model_\star - \model$. Suppose that
$
\norm{\Delta V^{1/2}} 
\le
\pOptimism \brk{2 \dx (\dx + \du) H^2}^{-1/2}
,
$
then, using the Lipschitz property, we have
\begin{align*}
    \abs{
    F(M) - F(M; \model)
    }
    &
    \le
    \EE[\zt, \ww]
    \abs*{
    c_t(x_t(M; \model_\star, \ww), u_t(M; \ww))
    -
    c_t(x_t(M; \model, \ww), u_t(M; \ww))
    }
    \\
    &
    \le
    \EE[\zt, \ww]
    \norm{
    x_t(M; \model, \ww) 
    -
    x_t(M; \model_\star, \ww)}
    \\
    &
    =
    \EE[\ww] \norm{\Delta \obs_{t-1}(M; \ww)}
    \\
    \tag{Cauchy-Schwarz}
    &
    \le
    \norm{\Delta V^{1/2}}
    \EE[\ww] \norm{V^{-1/2} \obs_{t-1}(M; \ww)}
    \\
    \tag{Jensen}
    &
    \le
    \norm{\Delta V^{1/2}}
    \sqrt{\EE[\ww] \obs_{t-1}(M; \ww)\tran V^{-1} \obs_{t-1}(M; \ww)}
    \\
    \tag{\cref{eq:bounded-mem-rep}}
    &
    =
    \norm{\Delta V^{1/2}}
    \sqrt{\tr{\obsOp(M)\tran V^{-1} \obsOp(M) \EE[\ww]
    \brk[s]{w_{t-2H:t-2} w_{t-2H:t-2}\tran}
    }}
    \\
    \tag{$\tr{AB} \le \tr{A} \norm{B}$}
    &
    \le
    \maxNoise \norm{\Delta V^{1/2}}
    \sqrt{\tr{\obsOp(M)\tran V^{-1} \obsOp(M)}}
    \\
    &
    =
    \maxNoise \norm{\Delta V^{1/2}}
    \norm{V^{-1/2} \obsOp(M)}_F
    \\
    \tag{$\norm{x}_2 \le \sqrt{d} \norm{x}_\infty$}
    &
    \le
    \maxNoise \sqrt{2 \dx (\dx + \du) H^2} \norm{\Delta V^{1/2}}
    \norm{V^{-1/2} \obsOp(M)}_\infty
    \\
    &
    \le
    \pOptimism \maxNoise
    \norm{V^{-1/2} \obsOp(M)}_\infty
    .
\end{align*}
Now, since $\model_{\tauIt}$ is assumed to satisfy the above condition, we get that
\begin{align*}
    F(M)
    \ge
    F(M; \model_{\tauIt})
    -
    \pOptimism \maxNoise \norm{V_{\tauIt}^{-1/2} \obsOp(M)}_\infty
    =
    \bar{F}_t(M)
    ,
\end{align*}
and on the other hand
\begin{align*}
    F(M)
    &
    \le
    F(M; \model_{\tauIt})
    +
    \pOptimism \maxNoise \norm{V_{\tauIt}^{-1/2} \obsOp(M)}_\infty
    \\
    &
    =
    \bar{F}_t(M)
    +
    2 \pOptimism \maxNoise \norm{V_{\tauIt}^{-1/2} \obsOp(M)}_\infty
    \\
    &
    \le
    \bar{F}_t(M)
    +
    2 \pOptimism \maxNoise \norm{V_{\tauIt}^{-1/2} \obsOp(M)}_F
    .
\end{align*}
Now, if $t \ge \tauIt + 2H$ then $V_{\tauIt}$ is independent of $w_{t-2H:t-2}$. Next, let 
$
\EE[\ww] \brk[s]{w_{t-2H:t-2} w_{t-2H:t-2}\tran}
$
$
=
\Sigma
$
and notice that the minimum covariance assumption implies that $\norm{\Sigma^{-1/2}} \le \wMin^{-1}$.
We thus have
\begin{align*}
    2 \pOptimism \maxNoise \norm{V_{\tauIt}^{-1/2} \obsOp(M)}_F
    &
    =
    2 \pOptimism \maxNoise \sqrt{\tr{\obsOp(M)\tran V_{\tauIt}^{-1} \obsOp(M)}}
    \\
    &
    \le
    2 \pOptimism \maxNoise \wMin^{-1}
    \sqrt{\tr{\obsOp(M)\tran V_{\tauIt}^{-1} \obsOp(M)\EE[\ww] \brk[s]{w_{t-2H:t-2} w_{t-2H:t-2}\tran}}}
    \\
    \tag{\cref{eq:bounded-mem-rep}}
    &
    =
    2 \pOptimism \maxNoise \wMin^{-1}
    \sqrt{\EE[\ww]\norm{V_{\tauIt}^{-1/2} \obs_{t-1}(M; \ww)}^2}
    .
\end{align*}
\end{proof}

\begin{lemma}
\label{lemma:regret-final-bound}
We have that with probability at least $1-\delta$
\begin{align*}
    \mathrm{regret}_T(\pi)
    &
    \le
    2860  
    \kappa^3 \gamma^{-11/2} \wMin^{-1} \maxNoise^2 \RM^2 \Bbound^2 
    (\dx^2 \kappa^2 + \du \Bbound^2)
    \sqrt{
    T \dx (\dx + \du)^3 
    }
    \log^5 \frac{18T^5}{\delta}
    \\
    &
    +
    12\costVar
    \sqrt{T \gamma^{-3} (\dx + \du) (\dx^2 \kappa^2 + \du^2 \Bbound^2)}
    \log^3 \frac{18 T^5}{\delta}
\end{align*}
\end{lemma}
\begin{proof}
Suppose that the events of \cref{lemma:lqrR15,lemma:lqrR24,lemma:lqrR3} hold. By a union bound, this holds with probability at least $1-\delta$.
Now, we simplify each of the terms before deriving the final bound.
Recall from \cref{thm:lqrRegretFull} that
\begin{align*}
    \pOptimism
    &
    =
    30 \maxNoise \RM \Bbound \kappa^2 (\dx + \du) H^2 
    \sqrt{\dx \gamma^{-3} (\dx^2 \kappa^2 + \du \Bbound^2) \log \frac{12T}{\delta}}
    \\
    &
    \le
    30 \maxNoise \RM \Bbound \kappa^2 (\dx + \du)
    \sqrt{\dx \gamma^{-7} (\dx^2 \kappa^2 + \du \Bbound^2) \log^5 \frac{12T}{\delta}}
\end{align*}
Also using the fact that $H \ge 2$ we get
\begin{align*}
    \ocoR{2} + \ocoR{3}
    &
    \le
    20 \pOptimism \kappa \gamma^{-1} \wMin^{-1} \maxNoise \RM \Bbound H
    \sqrt{
    T (\dx + \du) (\dx^2 \kappa^2 + \du \Bbound^2) \log^3 \frac{12T}{\delta}
    }
    \\
    &
    +
    104 \pOptimism \kappa \gamma^{-1} \RM \Bbound \sqrt{T H \du (\dx^2 \kappa^2 + \du \Bbound^2) \log^3 \frac{18 T^5}{\delta}}
    +
    16 \costVar \sqrt{
    T
    H^3 \dx \du \log^3 \frac{18 T^5}{\delta}
    }
    \\
    &
    \le
    94 \pOptimism \kappa \gamma^{-1} \wMin^{-1} \maxNoise \RM \Bbound H
    \sqrt{
    T (\dx + \du) (\dx^2 \kappa^2 + \du \Bbound^2) \log^3 \frac{18T^5}{\delta}
    }
    \\
    &
    +
    16 \costVar \sqrt{
    T
    H^3 \dx \du \log^3 \frac{18 T^5}{\delta}
    }
    \\
    &
    \le
    2820  
    \kappa^3 \gamma^{-11/2} \wMin^{-1} \maxNoise^2 \RM^2 \Bbound^2 
    (\dx^2 \kappa^2 + \du \Bbound^2)
    \sqrt{
    T \dx (\dx + \du)^3 
    }
    \log^5 \frac{18T^5}{\delta}
    \\
    &
    +
    16 \costVar \sqrt{
    T
    \gamma^{-3}
    \dx \du
    }
    \log^3 \frac{18 T^5}{\delta}
    .
\end{align*}
Next, we have
\begin{align*}
\ocoR{1} + \ocoR{5}
&
\le
(2\costVar
+
7 \kappa^2 \gamma^{-2} \Bbound^2 \RM^2 \maxNoise \sqrt{H})
\sqrt{32 T H^3 (\dx + \du) (\dx^2 \kappa^2 + \du^2 \Bbound^2) \log^3 \frac{18 T^4}{\delta}}
\\
&
\le
40 \kappa^2 \gamma^{-4} \Bbound^2 \RM^2 \maxNoise 
\sqrt{T (\dx + \du) (\dx^2 \kappa^2 + \du^2 \Bbound^2) \log^7 \frac{18 T^5}{\delta}}
\\
&
+
12\costVar
\sqrt{T \gamma^{-3} (\dx + \du) (\dx^2 \kappa^2 + \du^2 \Bbound^2)}
\log^3 \frac{18 T^5}{\delta}
.
\end{align*}
Combining both bounds, we conclude that
\begin{align*}
    \mathrm{regret}_T(\pi)
    &
    \le
    2860  
    \kappa^3 \gamma^{-11/2} \wMin^{-1} \maxNoise^2 \RM^2 \Bbound^2 
    (\dx^2 \kappa^2 + \du \Bbound^2)
    \sqrt{
    T \dx (\dx + \du)^3 
    }
    \log^5 \frac{18T^5}{\delta}
    \\
    &
    +
    12\costVar
    \sqrt{T \gamma^{-3} (\dx + \du) (\dx^2 \kappa^2 + \du^2 \Bbound^2)}
    \log^3 \frac{18 T^5}{\delta}
    .
    \qedhere
\end{align*}
\end{proof}

\section{Technical Lemmas and Proofs}
\label{sec:technicalProofs}

\subsection{Algebraic Lemmas}
The following is a statement of Lemma 27 of \cite{cohen2019learning}.
\begin{lemma}
\label{lemma:vTau-to-vt}
Let $V_1 \succeq V_2 \succeq 0$ be matrices in $\RR[d \times d]$, then we have 
\begin{align*}
    \obs\tran V_1 \obs 
    \le
    (\obs\tran V_2 \obs) \det(V_1) / \det(V_2)
    \qquad
    ,
    \forall \obs \in \RR[d]
    .
\end{align*}
\end{lemma}

Next, the following is a standard bound on a harmonic sum.
\begin{lemma*}[restatement of \cref{lemma:harmonicBound}]
    Let $\at \in \RR[\din]$ be a sequence such that $\norm{\at}^2 \le \lambda$, and define $V_t = \lambda I + \sum_{s=1}^{t-1} \at[s] \at[s]\tran$. Then
    $
        \sum_{t=1}^{T} \at\tran V_t^{-1} \at
        \le
        5 \din \log T
        .
    $
\end{lemma*}

\begin{proof}%
Notice that $\at\tran V_t^{-1} \at \le \norm{\at}^2 / \lambda^2 \le 1$, and so by Lemma 26 in \cite{cohen2019learning} we get that $\at V_t^{-1} \at \le \log \brk*{\det(V_{t+1}) / \det(V_t)}$. We conclude that
\begin{align*}
    \sum_{t=1}^{T} \at\tran V^{-1} \at
    &\le
    2 \sum_{t=1}^{T} \at\tran V_t^{-1} \at
    \\
    &
    \le
    4 \sum_{t=1}^{T} \log \brk*{\det(V_{t+1}) / \det(V_t)}
    \\
    \tag{telescoping sum}
    &
    =
    4 \log \brk*{\det(V_{T+1}) / \det(V)}
    \\
    &
    =
    4 \log \det(V^{-1/2} V_{T+1} V^{-1/2})
    \\
    \tag{$\det(A) \le \norm{A}^d$}
    &
    \le
    4 \din \log \norm{V^{-1/2} V_{T+1} V^{-1/2}}
    \\
    \tag{triangle inequality}
    &
    \le
    4 \din \log \brk*{1 + \frac{1}{\lambda^2}\sum_{s=1}^{T} \norm{\at[s]}^2}
    \\
    &
    \le
    4 \din \log (T+1)
    \\
    \tag{$T \ge 4$}
    &
    \le
    5 \din \log T
    .
\end{align*}
\end{proof}

\subsection{Concentration of Measure}

First, we give the following Bernstein type tail bound \cite[see e.g.,][Lemma D.4]{rosenberg2020near}.
\begin{lemma}
\label{lemma:multiplicative-concentration}
Let $\brk[c]{X_t}_{t \ge 1}$
be a sequence of random variables with expectation adapted to a filtration
$\mathcal{F}_t$.
Suppose that $0 \le X_t \le 1$ almost surely. Then with probability at least $1-\delta$
\begin{align*}
    \sum_{t=1}^{T} \EE \brk[s]{X_t \mid \mathcal{F}_{t-1}}
    \le
    2 \sum_{t=1}^{T} X_t
    +
    4 \log \frac{2}{\delta}
\end{align*}
\end{lemma}

\begin{lemma*}[restatement of \cref{lemma:blockConcentration}]
Let $X_t$ be a sequence of random variables adapted to a filtration $\mathcal{F}_t$. If $0 \le X_t \le 1$ then with probability at least $1 - \delta$ simultaneously for all $1 \le t \le T$
\begin{align*}
    \sum_{s=1}^{t} \EE\brk[s]{X_t \mid \mathcal{F}_{t-2H}}
    \le
    2 \sum_{s=1}^{t} (X_s)
    +
    8 H \log \frac{2 T^2}{\delta}
    .
\end{align*}
\end{lemma*}

\begin{proof}
For $h = 1, \ldots, 2H$, and $k \ge 0$ define the time indices
\begin{align*}
    t_{k}^{(h)}
    =
    h + 2H k
    =
    t_{k-1}^{(h)} + 2H
    ,
\end{align*}
and the filtration 
$
    \bar{\mathcal{F}}_k^{(h)} 
    =
    \mathcal{F}_{t_{k}^{(h)}}
    .
$
Denoting $X_k^{(h)} = X_{t_k^{(h)}}$ we have that $X_k^{(h)}$ is $\bar{\mathcal{F}}_k^{(h)}$ measurable, and thus $X_k^{(h)}$ satisfies \cref{lemma:multiplicative-concentration}, which we invoke with $\delta / 2H$ for all  $h = 1, \ldots, 2H$. Taking a union bound, we get that with probability at least $1 - \delta$ for all $h = 1, \ldots, 2H$
\begin{align*}
    \sum_{k=1}^{K(h)} 
    \EE\brk[s]{X_k^{(h)} \mid \bar{\mathcal{F}}_{k-1}^{(h)}}
    \le
    2 \sum_{k=1}^{K(h)} (X_k^{(h)})
    +
    4 \log \frac{4H}{\delta}
    \le
    2 \sum_{k=1}^{K(h)} (X_k^{(h)})
    +
    4 \log \frac{2T}{\delta}
    \tag{$2H \le T$}
    ,
\end{align*}
and thus
\begin{align*}
    \sum_{t=1}^{T} \EE\brk[s]{X_t \mid \mathcal{F}_{t-2H}}
    &
    =
    \sum_{h=1}^{2H}
    \sum_{k=1}^{K(h)}
    \EE\brk[s]{X_k^{(h)} \mid \bar{\mathcal{F}}_{k-1}^{(h)}}
    \\
    &
    \le
    \sum_{h=1}^{2H}
    \brk[s]*{
    2 \sum_{k=1}^{K(h)} \brk{X_k^{(h)}}
    +
    4 \log \frac{2T}{\delta}
    }
    \\
    &
    \le
    2 \sum_{t=1}^{T} (X_t)
    +
    8 H \log \frac{2T}{\delta}
    .
\end{align*}
Replacing $T$ with $t$ and $\delta$ with $\delta / T$, and taking a union bound over all $1 \le t \le T$ concludes the proof.
\end{proof}

\begin{lemma}[restatement of \cref{lemma:ocoUniformConvergence}]
\label{lemma:baseUC}
    Let $R > 0$ and suppose that 
    $
    T
    \ge
    64 R^2
    .
    $
    Then for any $\delta \in (0, 1)$ we have that with probability at least $1 - \delta$
    \begin{align*}
        \abs*{
        \sum_{t = 1}^{T}
        \lt(\qq)
        -
        \LL(\qq)
        }
        \le
        \ocoCostVar \sqrt{T \dout \log\brk*{\frac{3}{\delta}\max\brk[c]{1, \ocoCostVar^{-1} T}}}
        ,
        \quad
        \forall \qq \in \RR[\dout]
        \;
        \text{s.t.}
        \;
        \norm{\qq} \le R.
    \end{align*}
    If additionally $\ocoCostVar \ge 1$ then the $\log$ term may be bounded by $\log (3T / \delta)$.
\end{lemma}

\begin{proof}%
    Let 
    $
    \epsilon
    =
    \min\brk[c]{
    R
    ,
    \frac18 \ocoCostVar T^{-1/2}
    }
    ,
    $
    and $\mathcal{Q}_{\epsilon}$ be an $\epsilon-$cover of $B_{\dout}(R)$, the Euclidean norm ball in $\RR[\dout]$. It is well known that
    $
        \abs{\mathcal{Q}_\epsilon}
        \le
        (3 R / \epsilon)^{\dout}
    $
    \citep[see, e.g.,][Lemma 5.13]{van2014probability}.
    We can thus apply Hoeffding's inequality together with a union bound over $\qq \in \mathcal{Q}_\epsilon$ to get that with probability at least $1 - \delta$
    \begin{align*}
        \abs*{
        \sum_{t = 1}^{T}
        \lt(\qq) - \LL(\qq)
        }
        \le
        \ocoCostVar \sqrt{\frac12 T \dout \log \frac{3 R}{\epsilon \delta}}
        ,
        \quad
        \forall \qq \in \mathcal{Q}_\epsilon
        .
    \end{align*}
    Finally, assuming the above event holds, let $\qq \in B_{\dout}(R)$ be arbitrary and $\qq' \in \mathcal{Q}_\epsilon$ be its nearest point in the cover. Then, using the Lipschitz property of $\LL$ we conclude that
    \begin{align*}
        \abs*{
        \sum_{t = 1}^{T}
        \lt(\qq)
        -
        \LL(\qq)
        }
        &
        \le
        \abs*{
        \sum_{t = 1}^{T}
        \lt(\qq)
        -
        \lt(\qq')
        }
        +
        \abs*{
        \sum_{t = 1}^{T}
        \lt(\qq')
        -
        \LL(q')
        }
        +
        \abs*{
        \sum_{t = 1}^{T}
        \LL(\qq')
        -
        \LL(\qq)
        }
        \\
        &
        \le
        \ocoCostVar \sqrt{\frac12 T \dout \log \frac{3 R}{\epsilon \delta}}
        +
        2 T \epsilon
        \\
        &
        \le
        \ocoCostVar \sqrt{T \dout \log \frac{3 R}{\epsilon \delta}}
        \\
        &
        \le
        \ocoCostVar \sqrt{T \dout \log \brk*{\frac{3}{\delta}\max\brk[c]{1, \ocoCostVar^{-1} T}}}
        ,
    \end{align*}
    where the last two transitions used our choice of $\epsilon, T$ together with
    \begin{align*}
        \frac{R}{\epsilon} 
        =
        \max\brk[c]*{1, \frac{8 R \sqrt{T}}{\ocoCostVar}}
        \le
        \max\brk[c]*{1, \frac{T}{\ocoCostVar}}
        . 
        \quad
        &&&
        \qedhere
    \end{align*}
\end{proof}

\begin{lemma*}[restatement of \cref{lemma:blockUC}]
Let $f_t : \RR[d] \to \RR$ be a sequence of identically distributed functions such that $f_t$ and $f_{t+2H}$ are interdependently distributed. Let $F(M) = \EE f_t(M)$, $R > 0$ and suppose that $T \ge {64 R^2}$ where $C > 1$ is such that
\begin{align*}
    \abs{f_t(M) - f_{t+2H}(M)}
    \le
    C
    \;\;
    ,
    \forall \norm{M} \le R
    .
\end{align*}
The with probability at least $1-\delta$
\begin{align*}
    \abs*{\sum_{t=1}^{T} (f_t(M) - F(M))}
    \le
    C \sqrt{2 T H d \log \frac{3 T^2}{\delta}}
\end{align*}
\end{lemma*}
\begin{proof}
For $h = 1, \ldots, 2H$, and $k \ge 0$ define the time indices
\begin{align*}
    t_{k}^{(h)}
    =
    h + 2H k
    =
    t_{k-1}^{(h)} + 2H
    .
\end{align*}
Denoting $f_k^{(h)} = f_{t_k^{(h)}}$ we have that $f_k^{(h)}$ are i.i.d.
We can thus invoke \cref{lemma:ocoUniformConvergence} with a union bound over all $h = 1, \ldots, 2H \le T$
to get that with probability at least $1 - \delta$
\begin{align*}
    \sum_{k=1}^{K(h)} 
    \brk*{f_k^{(h)}(M) - F(M)}
    \le
    C \sqrt{K(h) d \log \frac{3 T^2}{C \delta}}
    ,
    \;\;
    \forall M \text{ s.t. } \norm{M} \le R \text{ and } h = 1, \ldots, 2H
    .
\end{align*}
where we denoted $K(h) = \floor{(T-h)/2H}$. Now, notice that \begin{align*}
\brk[c]{t_k : k = 1, \ldots, K(h), h = 1, \ldots, 2H} 
=
\brk[c]{1, \ldots, T}
.
\end{align*}
We conclude that
\begin{align*}
    \sum_{k=1}^{K(h)} 
    \brk*{f_k^{(h)}(M) - F(M)}
    \le
    C \sqrt{K(h) d \log \frac{3 T^2}{C \delta}}
    ,
    \;\;
    \forall M \text{ s.t. } \norm{M} \le R \text{ and } h = 1, \ldots, 2H
    .
\end{align*}
\begin{align*}
    \abs*{
    \sum_{t=1}^{T} \brk*{f_t(M) - F(M)}
    }
    &
    =
    \abs*{
    \sum_{h=1}^{2H}
    \sum_{k=1}^{K(h)}
    \brk*{f_k^{(h)}(M) - F(M)}
    }
    \\
    &
    \le
    \sum_{h=1}^{2H}
    C \sqrt{K(h) d \log \frac{3 T^2}{C \delta}}
    \\
    &
    \le
    C \sqrt{2H \sum_{h=1}^{2H}K(h) d \log \frac{3 T^2}{C \delta}}
    \\
    &
    =
    C \sqrt{2 T H d \log \frac{3 T^2}{C \delta}}
    ,
\end{align*}
for all $M$ such that $\norm{M} \le R$.
\end{proof}

\subsection{Least Squares Estimation}

Our algorithms use regularized least squares methods in order to estimate the system parameters.
An analysis of this method for a general, possibly-correlated sample, was introduced in the context of linear bandit optimization~\citep{abbasi2011improved}, and was first used in the context of LQRs by~\citet{abbasi2011regret}.

Let $\model_\star \in \RR[d \times m]$, $\seqDef{y_{t+1}}{t=1}{\infty} \in \RR[d]$, $\seqDef{z_t}{t=1}{\infty} \in \RR[m]$, $\seqDef{w_t}{t=1}{\infty} \in \RR[d]$, such that
$y_{t+1} = \model_\star z_t + w_t$, and $\seqDef{w_t}{t=1}{\infty}$ are i.i.d.\ and satisfy $\norm{w_t}_\infty \le \maxNoise$. 
Moreover, there exists a filtration $\brk[c]{\mathcal{F}_t}_{t \ge 1}$ such that $y_t, z_t$ are $\mathcal{F}_{t-1}$ measurable, and $w_t$ is $\mathcal{F}_t$ measurable and satisfies
$
\EE \brk[s]{w_t \; | \; \mathcal{F}_{t-1}} = 0
.
$
Denote by
\begin{align}
\label{eq:olsDef}
    \hat{\model}_t
    \in
    \argmin_{\model \in \RR[d \times m]} \left\{ \sum_{s=1}^{t-1} \norm{y_{t+1} - \model z_t}^2 + \lambda \norm{\model}_F^2 \right\}
    ,
\end{align}
the regularized least squares estimate of $\model_\star$ with regularization parameter $\lambda$.
The following lemma is due to \citealp{abbasi2011regret}.

\begin{lemma} 
\label{lemma:LSE}
Let 
$
V_t 
=
\lambda I + \sum_{s=1}^{t-1}z_t z_t\tran
$
and
$
\Delta_t
=
\model_\star - \hat{\model}_t
,
$
and suppose that $\norm{z_t}^2 \le \lambda$, $T \ge d$.
With probability at least $1 - \delta$, we have for all $1 \le t \le T$
\begin{align*}
    \norm{\Delta_t}_{V_t}^2
    \le
    \tr{\Delta_t\tran V_t \Delta_t}
    \le
    8 \maxNoise^2 d^2 \log \brk*{\frac{ T}{\delta}}
    +
    2\lambda \norm{\model_\star}_F^2
    .
\end{align*}
\end{lemma}

\paragraph{Disturbance estimation.}

\begin{proof}[of \cref{lemma:disturbanceEstimation}]
    First, since projection is a contraction operator we have that
    \begin{align*}
        \norm{w_t - \hat{w}_t}
        &
        =
        \norm*{
        \Pi_{B_2(\maxNoise)}\brk[s]{w_t}
        -
        \Pi_{B_2(\maxNoise)}\brk[s]{x_{t+1}- (A_t \; B_t) z_t}
        }
        \\
        &
        \le
        \norm*{
        {w_t}
        -
        {x_{t+1} + (A_t \; B_t) z_t}
        }
        \\
        &
        =
        \norm*{
        \brk[s]{(A_t \; B_t) - (\Astar \; \Bstar)} z_t
        }
        \\
        &
        \le
        \norm*{
        {(A_t \; B_t) - (\Astar \; \Bstar)}}_{V_t} \norm{V_t^{-1/2} z_t}
        ,
    \end{align*}
    where strictly for the purpose of this proof we denote $V_t = \pRegW I + \sum_{s=1}^{t-1} z_t z_t\tran$. In all other places we use $V_t$ as it is defined in \cref{alg:lqr}. Now, by \cref{lemma:technicalParameters} we have that
    \begin{align*}
        \norm{(x_t \; u_t)}^2
        =
        \norm{x_t}^2 + \norm{u_t}^2
        \le
        4 \kappa^2 \Bbound^2 \maxNoise^2 \RM^2 H \gamma^{-2}
        +
        \maxNoise^2 \RM^2 H
        \le
        5 \kappa^2 \Bbound^2 \maxNoise^2 \RM^2 H \gamma^{-2}
        =
        \pRegW
    \end{align*}
    We can thus invoke \cref{lemma:harmonicBound} to get that
    \begin{align*}
        \sum_{t=1}^{T} \norm{V_t^{-1/2} z_t}^2 \le 5 (\dx + \du) \log T
        .
    \end{align*}
    Finally, we have that
    \begin{align*}
        \norm{(\Astar \; \Bstar)}_F^2
        &
        =
        \tr{(\Astar \; \Bstar)\tran (\Astar \; \Bstar)}
        \\
        &
        =
        \tr{\Astar\tran \Astar} + \tr{\Bstar\tran \Bstar}
        \le
        \dx \norm{\Astar}^2 + \du \norm{\Bstar}^2
        \le
        \dx \kappa^2 + \du \Bbound^2
        .
    \end{align*}
    We can thus apply \cref{lemma:LSE} to get that with probability at least $1 - \delta$
    \begin{align*}
        \sum_{t=1}^{T}
        \norm{w_t - \hat{w}_t}^2
        &
        \le
        \sum_{t=1}^{T}
        \norm*{
        {(A_t \; B_t) - (\Astar \; \Bstar)}}_{V_t}^2 \norm{V_t^{-1/2} z_t}^2
        \\
        &
        \le
        \brk*{2\pRegW(\dx \kappa^2 + \du \Bbound^2) 
        +
        8 \maxNoise^2 \dx^2\log \frac{T}{\delta}}
        \sum_{t=1}^{T} \norm{V_t^{-1/2} z_t}^2
        \\
        &
        \le
        5(\dx + \du)\brk*{2\pRegW(\dx \kappa^2 + \du \Bbound^2)
        +
        8 \maxNoise^2 \dx^2\log \frac{T}{\delta}} \log T
        \\
        &
        \le
        5(\dx + \du) \maxNoise^2 H \brk*{10 \kappa^2 \RM^2 \Bbound^2 \gamma^{-2} (\dx \kappa^2 + \du \Bbound^2)
        +
        8 \dx^2} \log \frac{T}{\delta}
        \\
        &
        \le
        100 \maxNoise^2 H { \kappa^2 \RM^2 \Bbound^2 \gamma^{-2} (\dx + \du) (\dx^2 \kappa^2 + \du \Bbound^2)} \log \frac{T}{\delta}
        ,
    \end{align*}
    where the second to last inequality also used the fact that $H \ge \log T$.
\end{proof}

\paragraph{Model estimation.}

\begin{proof}[of \cref{lemma:lqrParameterEst}]
Recall that by recursively unrolling the transition model we get that
\begin{align*}
    x_{t+1}
    =
    \model_\star \obs_{t} + {w}_{t} + e_{t}
    ,
\end{align*}
where
$
    e_{t}
    =
    \Astar^{H} x_{t+1-H}
    +
    \sum_{h=1}^{H} \Astar^{h-1} ({w}_{t+1-h} - \hat{w}_{t+1-h})
    .
$
Moreover, letting $O_t, X_t$ be a matrices whose rows are $\obs_1, \ldots, \obs_{t-1}$, and $x_2, \ldots, x_t$ correspondingly. It is well known that the solution to the regularized least squares problem in \cref{alg:lqr} can be written as
$
\model_t
=
V_t^{-1} O_t\tran X_t
,
$
where, as in \cref{alg:lqr}, $V_t = \pRegTheta I + O_t\tran O_t.$
Now, define $\tilde{x}_{t+1} = \model_\star \obs_t + w_t$ and let $\tilde{\model}_t$ be it least squares solution, i.e.,
\begin{align*}
    \tilde{\model}_t
    =
    \argmin_{\model \in \RR[\dx \times H \du + (H-1)\dx]} \left\{ \sum_{s=1}^{t-1} \norm{\tilde{x}_{t+1} - \model \obs_t}^2 + \pRegTheta \norm{\model}_F^2 \right\}
    =
    V_t^{-1} O_t\tran \tilde{X}_t
    ,
\end{align*}
where and $\tilde{X}_t$ is a matrix whose rows are $\tilde{x}_2, \ldots, \tilde{x}_t$.
Notice that this fits the setting of \cref{lemma:LSE} and thus with probability at least $1 - \delta$
\begin{align*}
    \tr{(\tilde{\model}_t - \model_\star)\tran V_t (\tilde{\model}_t - \model_\star)}
    \le
    8 \maxNoise^2 \dx^2 \log \brk*{\frac{T}{\delta}}
    +
    2\pRegTheta \norm{\model_\star}_F^2
    ,
    \;\;
    \forall 1 \le t \le T
    .
\end{align*}
Next, denote $E_t$ the matrix whose rows are $e_1, \ldots, e_{t-1}$ and notice that $X_t = \tilde{X}_t + E_t$. We thus have that
\begin{align*}
    \tr{(\model_t - \tilde{\model}_t)\tran V_t (\model_t - \tilde{\model}_t)}
    &
    =
    \tr{E_t\tran O_t V_t^{-1} V_t V_t^{-1} O_t\tran E_t}
    \\
    &
    =
    \tr{E_t\tran O_t V_t^{-1} O_t\tran E_t}
    \le
    \tr{E_t\tran E_t}
    =
    \sum_{s=1}^{t-1} \norm{e_s}^2
    ,
\end{align*}
where the inequality holds since $V_t = O_t\tran O_t + \pRegTheta I$ and thus $\norm{O_t V_t^{-1} O_t} \le 1$.
Finally, combining this with the known result for the error bound on $\tilde{\model}_t$ concludes the proof.
Combining with the previous inequality we get that with probability at least $1 - \delta$
\begin{align*}
    \tr{(\model_t - \model_\star)\tran V_t (\model_t - \model_\star)}
    &
    \le
    2 \tr{(\tilde{\model}_t - \model_\star)\tran V_t (\tilde{\model}_t - \model_\star)}
    +
    2 \tr{(\model_t - \tilde{\model}_t)\tran V_t (\model_t - \tilde{\model}_t)}
    \\
    &
    \le
    16 \maxNoise^2 \dx^2 \log \brk*{\frac{T}{\delta}}
    +
    4\pRegTheta \norm{\model_\star}_F^2
    +
    2\sum_{s=1}^{t-1} \norm{e_s}^2
    ,
\end{align*}
for all $1 \le t \le T$. This concludes the first part of the proof.

For the second part, recall that we also assume that
$
\pRegTheta
=
2 \maxNoise^2 \RM^2 H^2
,
$
and $\sum_{t=1}^{T} \norm{\hat{w}_t - w_t}^2 \le \wErr^2$ where 
\begin{align*}
    \wErr
    =
    10 \maxNoise \kappa \RM \Bbound \gamma^{-1} \sqrt{H (\dx + \du) (\dx^2 \kappa^2 + \du \Bbound^2) \log \frac{T}{\delta}}
    .
\end{align*}
Since by \cref{lemma:technicalParameters}
$
\norm{x_t}
\le
2 \kappa \Bbound \maxNoise \RM \sqrt{H} / \gamma
,
$
and since
$
    H
    \ge
    \gamma^{-1} \log T
    ,
$
we have that
\begin{align*}
    \tag{triangle inequality}
    \norm{e_{t-1}}
    &
    \le
    \norm{\Astar^H x_{t-H}}
    +
    \sum_{h=1}^{H} \norm{\Astar^{h-1} (w_{t-h} - \hat{w}_{t-h})}
    \\
    \tag{sub-multiplicativity}
    &
    \le
    \norm{\Astar^H} \norm{x_{t-H}}
    +
    \sum_{h=1}^{H} \norm{\Astar^{h-1}} \norm{w_{t-h} - \hat{w}_{t-h}}
    \\
    \tag{Cauchy-Schwarz}
    &
    \le
    2 \kappa^2 \gamma^{-1} \Bbound \maxNoise \RM \sqrt{H} e^{-\gamma H}
    +
    \sqrt{
    \sum_{h=1}^{H} \norm{\Astar^{h-1}}^{2}
    }
    \sqrt{
    \sum_{h=1}^{H} \norm{w_{t-h} - \hat{w}_{t-h}}^2
    }
    \\
    \tag{strong stability}
    &
    \le
    2 \kappa^2 \gamma^{-1} \Bbound \maxNoise \RM \sqrt{H} T^{-1}
    +
    \sqrt{\kappa^2 \gamma^{-1}
    \sum_{h=1}^{H} \norm{w_{t-h} - \hat{w}_{t-h}}^2
    }
    .
\end{align*}
Taking the square and summing over $t$, we get that
\begin{align*}
\tag{$(x+y)^2 \le 2(x^2 + y^2)$}
    \sum_{t=1}^{T} \norm{e_t}^2
    &
    \le
    4 \kappa^4 \maxNoise^2 \RM^2 \Bbound^2 H \gamma^{-2} T^{-1}
    +
    2 \kappa^2 \gamma^{-1} H \sum_{t=1}^{T} \norm{w_{t} - \hat{w}_{t}}^2
    \\
    &
    \le
    4 \kappa^4 \maxNoise^2 \RM^2 \Bbound^2 H \gamma^{-2}
    +
    2 \kappa^2 \gamma^{-1} H \wErr^2
    \\
    &
    \le
    204 \maxNoise^2 \kappa^4 \RM^2 \Bbound^2 \gamma^{-3} H^2 (\dx + \du) (\dx^2 \kappa^2 + \du \Bbound^2) \log \frac{T}{\delta}
    .
\end{align*}
Now,
by \cref{lemma:technicalParameters} we also have that 
$
\norm{(\model_\star \; I)}_F \le \Bbound \kappa \sqrt{2 \dx / \gamma}
$
and thus plugging into the first part of the proof we get that
\begin{align*}
    \tr{\Delta_t\tran V_t \Delta_t}
    &
    \le
    16 \maxNoise^2 \dx^2 \log \brk*{\frac{T}{\delta}}
    +
    4\pRegTheta \norm{\model_\star}_F^2
    +
    2\sum_{s=1}^{t-1} \norm{e_s}^2
    \\
    &
    \le
    16 \maxNoise^2 \dx^2 \log \brk*{\frac{T}{\delta}}
    +
    16 \maxNoise^2 \kappa^2 \gamma^{-1} \RM^2 \Bbound^2 H^2 \dx
    \\
    &
    +
    408 \maxNoise^2 \kappa^4 \RM^2 \Bbound^2 \gamma^{-3} H^2 (\dx + \du) (\dx^2 \kappa^2 + \du \Bbound^2) \log \frac{T}{\delta}
    \\
    &
    \le
    441 \maxNoise^2 \kappa^4 \RM^2 \Bbound^2 \gamma^{-3} H^2 (\dx + \du) (\dx^2 \kappa^2 + \du \Bbound^2) \log \frac{T}{\delta}
    ,
\end{align*}
and taking the square root concludes the second part of the proof.
Finally, we can use the above to conclude that
\begin{align*}
    \norm{\brk{\model_t \; I}}_F
    &
    \le
    \norm{\brk{\model_\star \; I}}_F
    +
    \norm{\model_t - \model_\star}_F
    \\
    &
    \le
    \norm{\brk{\model_\star \; I}}_F
    +
    \frac{1}{\sqrt{\pRegTheta}}\norm{\model_t - \model_\star}_{V_t}
    \\
    &
    \le
    \Bbound \kappa \sqrt{2 \dx / \gamma}
    +
    15 \Bbound \kappa^2 
    \sqrt{\gamma^{-3} (\dx + \du) (\dx^2 \kappa^2 + \du \Bbound^2) \log \frac{T}{\delta}}
    \\
    &
    \le
    17 \Bbound \kappa^2 
    \sqrt{\gamma^{-3} (\dx + \du) (\dx^2 \kappa^2 + \du \Bbound^2) \log \frac{T}{\delta}}
    ,
\end{align*}
as desired.

\end{proof}

\subsection{Surrogate functions}

\begin{proof}[of \cref{lemma:technicalParameters}]
First, notice that
\begin{align*}
    \norm{\brk*{\model_\star \; I}}_F^2
    &
    =
    \sum_{h=1}^{H} \tr{\Astar^{h-1}\Bstar\Bstar\tran {\Astar^{h-1}}\tran}
    +
    \tr{\Astar^{h-1}{\Astar^{h-1}}\tran}
    \\
    &
    \le
    2 \dx \Bbound^2 \sum_{h=1}^{H} \norm{\Astar^{h-1}}^2
    \\
    \tag{strong stability}
    &
    \le
    2 \dx \Bbound^2 \kappa^2 \sum_{h=1}^{H} (1-\gamma)^{2(h-1)}
    \\
    &
    \le
    2 \dx \Bbound^2 \kappa^2 \gamma^{-1}.
\end{align*}
Next, we have that
\begin{align*}
    \norm{u_t(M; \ww)}
    \le
    \sum_{h=1}^{H} \norm{M^{[h]}} \norm{w_{t-h}}
    \le
    \maxNoise \sum_{h=1}^{H} \norm{M^{[h]}}
    \le
    \maxNoise \RM \sqrt{H}
    ,
\end{align*}
where the last transition is due to Cauchy-Schwarz.
Next, we have that
\begin{align*}
    \norm{(\obs_t(M; \ww) \; w_{t})}
    =
    \sqrt{
    \sum_{h=1}^{H} \brk{\norm{u_{t+1-h}(M; \ww)}^2
    +
    \norm{w_{t+1-h}}^2}
    }
    \le
    \sqrt{2} \maxNoise \RM H
    .
\end{align*}
Next, we have that
\begin{align*}
    \norm{x_t(M; \model_\star, \ww)}
    &
    =
    \norm*{
    \sum_{h=1}^{H} \Astar^{h-1}\brk*{
    \Bstar u_{t-h}(M; \ww)
    +
    w_{t - h}}
    }
    \\
    &
    \le
    2 \kappa \Bbound \maxNoise \RM \sqrt{H}\sum_{h=1}^{H} (1-\gamma)^{h-1}
    \\
    &
    \le
    2 \kappa \Bbound \maxNoise \RM \sqrt{H} / \gamma,
\end{align*}
and $\norm{x_t},\norm{x_t^{\pi_M}}$ are bounded exactly the same but with $H = t-1$.
Next, we have that
\begin{align*}
    \norm{u_t(M; \ww) - u_t(M; \ww')}
    =
    \norm*{\sum_{h=1}^{H} M^{[h]} (w_{t-h} - w'_{t-h})}
    \le
    \RM \norm{w_{t-H:t-1} - w'_{t-H:t-1}}
    ,
\end{align*}
Finally, we have that
\begin{align*}
    &
    \sqrt{
    \norm{\obs_t(M; \ww) - \obs_t(M; \ww')}^2
    +
    \norm{w_t - w_t'}^2
    }
    \\
    &
    =
    \sqrt{
    \sum_{h=1}^{H} \brk*{
    \norm{u_{t+1-h}(M; \ww) - u_{t+1-h}(M; \ww')}^2
    +
    \norm{w_{t+1-h} - w_{t+1-h}}^2
    }
    }
    \\
    &
    \le
    \RM
    \sqrt{
    \sum_{h=1}^{H}
    \norm{w_{t+1-h-H:t+1-h} - w'_{t+1-h-H:t+1-h}}^2
    }
    \\
    &
    \le
    \RM \sqrt{H} \norm{w_{t+1-2H:t} - w_{t+1-2H:t}'}
    .
    \qedhere
\end{align*}
\end{proof}

\begin{proof}[of \cref{lemma:fbarProperties}]
First, recalling the definition of $x_t(M; \model, \ww)$ in \cref{eq:bounded-mem-rep}, we have
\begin{align*}
    x_t(M; \model, \ww)
    =
    \model_\star \obs_{t-1}(M; \ww) + w_{t-1}
    =
    \sum_{h=1}^{H} \Astar^{h-1}\brk[s]*{
    \Bstar u_{t-h}(M; \ww) + w_{t-h}}
    .
\end{align*}
Also noticing that $u_t(M;\ww)$ is $\RM$ Lipschitz in $w_{t-H:t-1}$ (\cref{lemma:technicalParameters}), we get
\begin{align*}
    &
    \norm{
    x_t(M; \model_\star, \ww)
    -
    x_t(M; \model_\star, \ww')
    }
    =
    \norm*{
    \sum_{h=1}^{H} \Astar^{h-1}\brk[s]*{
    \Bstar (u_{t-h}(M; \ww) - u_{t-h}(M; \ww'))
    +
    (w_{t - h} - w'_{t-h})}
    }
    \\
    &
    \le
    \sum_{h=1}^{H} \kappa (1-\gamma)^{h-1}\brk[s]*{
    \Bbound \norm{u_{t-h}(M; \ww) - u_{t-h}(M; \ww')}
    +
    \norm{w_{t - h} - w'_{t-h}}}
    \\
    &
    \le
    \sum_{h=1}^{H} \kappa (1-\gamma)^{h-1}\brk[s]*{
    \Bbound \RM \norm{w_{t-(h+H):t-(h+1)} - w_{t-(h+H):t-(h+1)}'}
    +
    \norm{w_{t - h} - w'_{t-h}}}
    \\
    \tag{${x} + {y} \le \sqrt{2(x^2 + y^2)}$}
    &
    \le
    \sqrt{2} \kappa \sum_{h=1}^{H}  (1-\gamma)^{h-1}
    \Bbound \RM \norm{w_{t-(h+H):t-h} - w_{t-(h+H):t-h}'}
    \\
    &
    \le
    \sqrt{2} \kappa \gamma^{-1}
    \Bbound \RM \norm{w_{t-2H:t-1} - w_{t-2H:t-1}'}
    ,
\end{align*}
where in the third inequality notice that $\norm{w_{1:t-1}}^2 + \norm{w_t}^2 = \norm{w_{1:t}}^2$. used 

Next, also using the Lipschitz assumption on $c_t$, and that $u_t$ is $\RM-$Lipschitz with respect to $w_{t-H:t-1}$ (\cref{lemma:technicalParameters}), we get that
\begin{align*}
    \|
    f_t(M; \model_\star, \ww, \zt[]) 
    &
    -
    f_t(M; \model_\star, \ww', \zt[]')
    \|
    \le
    2\costVar
    +
    \abs{f_t(M; \model_\star, \ww, \zt[]) - f_t(M; \model_\star, \ww', \zt[])}
    \\
    &=
    2\costVar
    +
    \abs{
    c_t(x_t(M; \model_\star, \ww), u_t(M; \ww))
    -
    c_t(x_t(M; \model_\star, \ww'), u_t(M; \ww'))
    }
    \\
    &
    \le
    2\costVar
    +
    \norm{
    (
    x_t(M; \model_\star, \ww) - x_t(M; \model_\star, \ww')
    ,
    u_t(M; \ww) - u_t(M; \ww')
    )
    }
    \\
    &
    \le
    2\costVar
    +
    \sqrt{3} \kappa \gamma^{-1}
    \Bbound \RM \norm{w_{t-2H:t-1} - w_{t-2H:t-1}'}
    \\
    &
    \le
    2\costVar
    +
    5 \kappa \gamma^{-1} \Bbound \RM \maxNoise \sqrt{H}
    ,
\end{align*}
where the last transition also used
$
\norm{w_{t-2H:t-1} - w'_{t-2H:t-1}}
\le
\maxNoise \sqrt{8H}
.
$

Now, also recall
$
x_t(M; \model, \ww)
= 
\model \obs_{t-1}(M; \ww) + w_{t-1}
$,
thus by \cref{lemma:technicalParameters} we have
\begin{alignat*}{2}
    &
    \norm{x_t(M; \model, \ww)}
    &&
    \le
    \sqrt{2} \maxNoise \RM H \norm{(\model \; I)}
    \\
    &
    \norm{\brk*{x_t(M; \model, \ww) \; u_t(M; \ww)}}
    &&
    \le
    \sqrt{3} \maxNoise \RM H \norm{(\model \; I)}
    \\
    &
    \norm{x_t(M; \model, \ww) - x_t(M; \model, \ww')}
    &&
    \le
    \sqrt{2H} \RM \norm{(\model \; I)} \norm{w_{t-2H:t-1} - w'_{t-2H:t-1}}
    .
\end{alignat*}
Notice that the confidence term $\pOptimism \maxNoise \norm{V^{-1/2} \obsOp(M)}_\infty$ is independent of $\ww, \zt[]$.
We thus have
\begin{align*}
    \|
    \bar{f}_t(M; \model, V, \ww, \zt[])
    &
    -
    \bar{f}_t(M; \model, V, \ww', \zt[])
    \|
    =
    \abs{
    f_t(M; \model, \ww, \zt[])
    -
    f_t(M; \model, \ww', \zt[])
    }
    \\
    &
    =
    \abs{
    c_t(x_t(M; \model, \ww), u_t(M; \ww); \zt[])
    -
    c_t(x_t(M; \model, \ww'), u_t(M; \ww'); \zt[])
    }
    \\
    &
    \le
    \norm{
    (
    x_t(M; \model, \ww) - x_t(M; \model, \ww')
    ,
    u_t(M; \ww) - u_t(M; \ww')
    )
    }
    \\
    &
    \le
    \sqrt{3H} \RM \norm{(\model \; I)} \norm{w_{t-2H:t-1} - w'_{t-2H:t-1}}
    \\
    &
    \le
    \pLipW(\model)
    \norm{w_{t-2H:t-1} - w'_{t-2H:t-1}}
    ,
\end{align*}
and thus we also have
\begin{align*}
    \abs{
    \bar{f}_t(M; \model, V, \ww, \zt[])
    -
    \bar{f}_t(M; \model, V, \ww', \zt[]')
    }
    &
    \le
    2\costVar
    +
    \abs{
    \bar{f}_t(M; \model, V, \ww, \zt[])
    -
    \bar{f}_t(M; \model, V, \ww', \zt[])
    }
    \\
    &
    \le
    2\costVar
    +
    \sqrt{3H} \RM \norm{(\model \; I)} \norm{w_{t-2H:t-1} - w'_{t-2H:t-1}}
    \\
    &
    \le
    2\costVar
    +
    5 \RM \maxNoise H \norm{(\model \; I)}
    \le
    \maxF(\model)
    ,
\end{align*}
where the last transition used 
$
\norm{w_{t-2H:t-1} - w'_{t-2H:t-1}}
\le
\maxNoise \sqrt{8H}
.
$
Finally, if 
\begin{align*}
    \norm{\brk{\model \; I}}_F
    \le
    17 \Bbound \kappa^2 
    \sqrt{\gamma^{-3} (\dx + \du) (\dx^2 \kappa^2 + \du \Bbound^2) \log \frac{12T}{\delta}}
    ,
\end{align*}
then we have that
\begin{align*}
    \pLipW(\model)
    &
    \le
    30 \RM \Bbound \kappa^2 
    \sqrt{H \gamma^{-3} (\dx + \du) (\dx^2 \kappa^2 + \du \Bbound^2) \log \frac{12T}{\delta}}
    \le
    \pOptimism / (\maxNoise \sqrt{\dx (\dx + \du) H^3})
    \\
    \maxF(\model)
    &
    \le
    2\costVar
    +
    85 \maxNoise \RM \Bbound \kappa^2 H \sqrt{\gamma^{-3} (\dx + \du) (\dx^2 \kappa^2 + \du \Bbound) \log \frac{12 T}{\delta}}
    \\
    &
    \le
    2\costVar
    +
    3 \pOptimism / (H \sqrt{\dx(\dx+\du)})
    .
    \qedhere
\end{align*}
\end{proof}

\end{document}